\documentclass[12pt]{article}
\usepackage{amsmath}

\usepackage{amsthm}
\usepackage{enumerate}
\usepackage{amssymb}
\usepackage{ytableau}
\usepackage{youngtab}
\usepackage{graphicx}
\usepackage{caption}
\usepackage[font=footnotesize]{subcaption}
\usepackage{hyperref}

\usepackage{bbm}
\usepackage{geometry}
\usepackage[font=footnotesize]{caption}

\newtheorem{theorem}{Theorem}[section]
\newtheorem{proposition}[theorem]{Proposition}
\newtheorem{lemma}[theorem]{Lemma}

\newtheorem{conjecture}[theorem]{Conjecture}
\newtheorem{definition}[theorem]{Definition}
\newtheorem{problem}[theorem]{Problem}

\newtheorem{question}[theorem]{Question}

\newcommand{\E}{\mathbb{E}}

\title{The length of an $s$-increasing sequence of $r$-tuples}
\author{W.T.Gowers and J.Long}
\begin{document}
\maketitle
\begin{abstract}
We prove a number of results related to a problem of Po-Shen Loh~\cite{inctrips}, which is equivalent to a problem in Ramsey theory. Let $a=(a_1,a_2,a_3)$ and $b=(b_1,b_2,b_3)$ be two triples of integers. Define $a$ to be \emph{2-less than} $b$ if $a_i<b_i$ for at least two values of $i$, and define a sequence $a^1,\dots,a^m$ of triples to be \emph{2-increasing} if $a^r$ is 2-less than $a^s$ whenever $r<s$. Loh asks how long a 2-increasing sequence can be if all the triples take values in $\{1,2,\dots,n\}$, and gives a $\log_*$ improvement over the trivial upper bound of $n^2$ by using the triangle removal lemma. In the other direction, a simple construction gives a lower bound of $n^{3/2}$. We look at this problem and a collection of generalizations, improving some of the known bounds, pointing out connections to other well known problems in extremal combinatorics, and asking a number of further questions.
\end{abstract}

\section{Introduction}

This paper concerns a deceptively simple problem formulated recently by Po-Shen Loh~\cite{inctrips}. As he put it in an interview \cite{lohinterview}, ``I thought it had to be trivial, it's so easy to describe, surely it will fall from some simple argument like the pigeonhole principle, and I will be done. I wasn't done in one hour, actually I'm still not done, and in fact there have been quite a few people who tried it and they also are not done." 

We too are not done, but we have made some partial progress. Along the way, like Loh, we have noticed interesting connections to other parts of combinatorics, which we shall describe later and which lend support to Loh's view that his problem is, despite its simplicity, a deep and interesting one. Two other recent papers about it are \cite{TWY} and \cite{W}.

\subsection{2-increasing sequences of triples}

We start by defining a simple relation on triples of integers.

\begin{definition} Let $a=(a_1,a_2,a_3)$ and $b=(b_1,b_2,b_3)$ be two triples of integers. Say that $a$ is \emph{2-less than} $b$, or $a<_2b$, if $a_i<b_i$ for at least two coordinates $i$. 
\end{definition}
\noindent For example, $(3,3,9)<_2(5,6,1)<_2(7,7,7)<_2(7,8,9)$, but $(1,2,3)$ is not 2-less than $(1,2,4)$. 

We think of this relation as a sort of ordering, even though in fact it is not, since it is not transitive: for instance $(1,2,3)<_2(2,3,1)<_2(3,1,2)<_2(1,2,3)$. (This is the Condorcet paradox, and indeed Loh notes connections between his problem and questions in voting theory.) With that in mind, we make a further definition.

\begin{definition}\label{2inc}
A sequence $(a^i)$ of integer triples is \emph{2-increasing} if for all $i<j$ we have $a^i<_2a^j$.
\end{definition}
\noindent Note that because of the lack of transitivity, this is strictly stronger than saying that $a^i<_2a^{i+1}$ for each $i$.

We are now ready to state Loh's problem. Here and throughout the paper we write $[n]$ for the set $\{1,2,\dots,n\}$. 
\begin{problem}\label{PS1}
For each $n$, let $F(n)$ be the maximal length of a 2-increasing sequence of triples with each coordinate belonging to $[n]$. How does $F(n)$ grow with $n$?
\end{problem}
An instructive example is the following sequence of length 8, which is of maximal length when $n=4$:
	\begin{center}
		(1,1,1)\\
		(1,2,2)\\
		(2,1,3)\\
		(2,2,4)\\
		(3,3,1)\\
		(3,4,2)\\
		(4,3,3)\\
		(4,4,4)\\
	\end{center}

The following proposition gives the easy bounds for general $n$.
\begin{proposition}\label{triv}
For all $n$ we have $F(n)\le n^2$. Moreover, whenever $n$ is a perfect square we have $F(n)\ge n^{3/2}$.
\end{proposition}
\begin{proof}
The upper bound follows from the trivial remark that in a set of more than $n^2$ triples with coordinates from $[n]$ we must have two triples that are equal in their first two coordinates, by the pigeon-hole principle. But neither of these is 2-less than the other.
	
For the lower bound, we generalize the construction used in the example above. Say $n=m^2$ is a perfect square. We let the sequence of first coordinates be $m$ consecutive copies of $1,\dots, m^2$. Then we let the sequence of second coordinates be $m$ consecutive copies of $1,\dots, m$, followed by $m$ copies of $m+1,\dots,2m$, etc, finishing with $m$ copies of $m^2-m+1,\dots, m^2$.  Finally, we let the sequence of third coordinates be $m$ consecutive 1s, followed by $m$ consecutive 2s, etc, finishing with $m$ consecutive $m^2$s. For example, with $n=9$ we have the construction
	\begin{center}
		1\,2\,3\,4\,5\,6\,7\,8\,9\,1\,2\,3\,4\,5\,6\,7\,8\,9\,1\,2\,3\,4\,5\,6\,7\,8\,9\\
		1\,2\,3\,1\,2\,3\,1\,2\,3\,4\,5\,6\,4\,5\,6\,4\,5\,6\,7\,8\,9\,7\,8\,9\,7\,8\,9\\
		1\,1\,1\,2\,2\,2\,3\,3\,3\,4\,4\,4\,5\,5\,5\,6\,6\,6\,7\,7\,7\,8\,8\,8\,9\,9\,9\\
	\end{center}
where to save space we have written the triples as columns rather than rows.
	
	It is easy to check that this gives a 2-increasing sequence, and it has length $m^3$, as required.
\end{proof}
Faced with the above bounds, it is natural to think that the lower bound is probably closer to the truth, since the remark giving the upper bound is very weak. However, the main result of Loh's paper may reduce one's confidence in this view. For use in the proof, and later in the paper, we make the following definition.

\begin{definition} Two triples $t_1$ and $t_2$ are \emph{2-comparable} if one of them is 2-less than the other. A set of triples is \emph{2-comparable} if any two of them are 2-comparable.
\end{definition}

\begin{proposition}\label{PS}
	$F(n)\le n^2/\exp(\Omega(\log^*(n)))$.
\end{proposition}
\begin{proof}
	Let $T=(t_i)$ be a 2-increasing sequence of triples taking values in $[n]$, and let $t_i=(a_i,b_i,c_i)$. Now construct a tripartite graph with vertex sets $A=B=C=[n]$ by taking each triple $t_i$ and thinking of it as a triangle with vertices $a_i\in A, b_i\in B$ and $c_i\in C$. That is, we put in the edges $a_ib_i, b_ic_i$ and $a_ic_i$. 
	
	Note that no two of these triangles can share an edge. For example, if the edges $a_ib_i$ and $a_jb_j$ are the same, then $a_i=a_j$ and $b_i=b_j$, which implies that neither of the triples $(a_i,b_i,c_i)$ and $(a_j,b_j,c_j)$ can be 2-less than the other. Furthermore, these are the only triangles in the graph, since if we have a triangle with all three of its edges coming from different triples, then we have three triples in our collection, of the form $(x,b,c), (a,y,c), (a,b,z)$, which must be 2-comparable. If $x<a$, then we can deduce from the 2-comparability that $b<y$, which in turn gives us that $c>z$, which then implies that $x>a$, a contradiction. Similarly, if $x>a$ we can deduce that $x<a$ and again obtain a contradiction.
	
	It follows that no two triangles in the graph we have just constructed share an edge. But by the triangle removal lemma \cite{RS}, any such graph has $o(n^2)$ edges, and using the best-known bounds, due to Fox \cite{Fox}, we obtain the result stated.
\end{proof}
After seeing this proof, one might now expect that the correct bound is of the form $n^{2-o(1)}$, with a lower bound provided by a suitable modification of Behrend's surprisingly dense set that contains no arithmetic progression of length 3 \cite{Behrend}. However, it does not take long to see that this does not work: in brief, the reason is that the 2-comparable and 2-increasing conditions impose far stronger constraints on the graph than the ones used in the above proof. (For more details, see Section 2.3 of Loh's paper.) 

We end the description of the problem with a simple product argument that shows that if for any fixed $k$ one could obtain any improvement at all over the lower bound of $k^{3/2}$, then we could deduce that asymptotically $F(n)$ beats $n^{3/2}$ in the exponent (meaning that there exists some $\alpha>3/2$ such that $F(n)>n^\alpha$ for all sufficiently large $n$).
\begin{lemma}\label{prod}
	Suppose that for some $n$ we have $F(n)=n^{\alpha}$. Then there are arbitrarily large $m$ such that $F(m)\ge m^\alpha$.
\end{lemma}
\begin{proof}
	We define the product $\otimes$ of two sequences in an obvious way: given two 2-increasing sequences $(a_i,b_i,c_i)$ and $(d_j,e_j,f_j)$, form a sequence $((a_i,d_j),(b_i,e_j),(c_i,f_j))$, where the indices $(i,j)$ are arranged lexicographically. Also, take the lexicographical ordering on the pairs themselves. Then if $(i,j)<(k,l)$ we either have $i<k$, in which case
	\[((a_i,d_j),(b_i,e_j),(c_i,f_j))<((a_k,d_l),(b_k,e_l),(c_k,f_l))\]
	just because $(a_i,b_i,c_i)<(a_k,b_k,c_k)$, or we have $i=k$ and $j<l$, in which case we are done because of the second coordinates. Finally, we can just inject pairs $(x,y)$ with $x,y\in [n]$ into $[n^2]$ with an injection that respects the lex ordering. So if we have a sequence of tuples $T$ with $|T|=n^\alpha$ then by taking $T\otimes\dots \otimes T$ we can boost the construction to arbitrarily large $m$.
\end{proof}

Observe also that since for every $m$ there is an integer power of $n$ that lies between $m/n$ and $m$, we can also deduce from the assumption of the lemma that $F(m)\geq(m/n)^\alpha$ for every $m$. Therefore, for every $\beta<\alpha$ and all sufficiently large $m$, we have that $F(m)\geq m^\beta$.

In the light of this result, it is natural to try a computer search to see whether it throws up any small examples that give rise to an exponent greater than 3/2. We have tried this and failed to find any, which lends some support to the following conjecture, which is also suggested by remarks that Loh makes in his paper.\begin{conjecture}\label{IncUB}
	$F(n)\le n^{3/2}$ for all $n$.	
\end{conjecture}

\subsection{Weakening the main condition to 2-comparability}
The proof of Loh's upper bound, Proposition \ref{PS}, did not make full use of the property that the sequence of triples is 2-increasing: all that was needed was that it was 2-comparable (recall that this means that for any two triples in the sequence, one is 2-less than the other). It is therefore natural to consider the following weakening of Problem~\ref{PS1}.

\begin{problem}\label{PS2}
	For each $n$, let $G(n)$ be the maximal size of a 2-comparable set of triples with each coordinate belonging to $[n]$. How does $G(n)$ grow with $n$?
\end{problem}

While this problem is no longer equivalent to the Ramsey question that motivated Loh, it too turns out to be surprisingly interesting. We shall discuss it further later in the paper, and provide some connections from this question to other problems in extremal combinatorics.

From the remarks we have just made, and the fact that $G(n)\geq F(n)$ for every $n$, we have the following result.
\begin{proposition}\label{triv3}
	For all $n$ we have $G(n)\le n^2/\exp(\Omega(\log^*(n)))$. Moreover, whenever $n=m^2$ we have $G(n)\ge n^{3/2}$.
\end{proposition}
\noindent Also, essentially the same product argument shows that Lemma \ref{prod} is true for $G$ just as it is for $F$.

Table \ref{tab1} gives the values of $F$ and $G$ for very small $n$, calculated by a brute-force computer search. (The number of 2-comparable sequences grows very rapidly with $n$, so such a search was not feasible for larger $n$ on the computer we used.) So at least for these $n$ weakening the assumption to 2-comparability does not lead to significant improvements over the construction outlined in Proposition \ref{triv}.

\begin{table}[]
	\centering
	\begin{tabular}{|c|ccccc|}
		\hline
		$n$                       & 1 & 2 & 3 & 4 & 5     \\
		$\lfloor n^{3/2} \rfloor$ & 1 & 2 & 5 & 8 & 11   \\
		$F(n)$                    & 1 & 2 & 4 & 8 & 10   \\
		$G(n)$                    & 1 & 2 & 5 & 8 & 11   \\
		\hline
	\end{tabular}
		\caption{Experimental results for small $n$}
		\label{tab1}
\end{table}

\subsection{Generalizing to $s$-increasing sequences of $r$-tuples}

It is natural to consider what happens if we generalize the problem in an obvious way from 2-increasing or 2-comparable sequences of triples to $s$-increasing or $s$-comparable sequences of $r$-tuples. So let us make the following definitions.

\begin{definition}
An $r$-tuple $a=(a_1,\dots,a_r)$ of integers is \emph{$s$-less than} an $r$-tuple $b=(b_1,\dots,b_r)$ if $a_i<b_i$ for at least $s$ values of $i$. In that case we write $a<_sb$. An $s$-\emph{increasing} sequence of $r$-tuples is a sequence $(a^1,\dots,a^m)$ such that $a^i<_sa^j$ whenever $i<j$. Two $r$-tuples are $s$-\emph{comparable} if one is $s$-less than the other, and an $s$-\emph{comparable} set of $r$-tuples is a set $\{a^1,\dots,a^m\}$ such that any two distinct elements of the set are $s$-comparable.
\end{definition}

It will be convenient to refer to an $s$-increasing sequence of $r$-tuples as an $(r,s)$-sequence and an $s$-comparable sequence of $r$-tuples as an $[r,s]$-sequence. 

Let $F_{r,s}(n)$ be the greatest possible length of an $(r,s)$-sequence and let $G_{r,s}(n)$ be the greatest possible length of an $[r,s]$-sequence such that the $r$-tuples take values in $[n]$. The following proposition generalizes Proposition \ref{triv}.
\begin{proposition}\label{triv2}
	For all $r,s$ and $n$ we have $F_{r,s}(n)\le G_{r,s}(n)\le n^{r-s+1}$. Moreover, whenever $n$ is a perfect $s$th power, we have $G_{r,s}(n)\ge F_{r,s}(n)\ge n^{r/s}$.
\end{proposition}

\begin{proof}
	As with Proposition \ref{triv}, the upper bound follows instantly from the pigeonhole principle. Also, it is trivial that $F_{r,s}(n)\leq G_{r,s}(n)$ for every $r,s$ and $n$.
		
	The lower bound is obtained by generalizing the construction in Proposition \ref{triv} in a straightforward, but not quite trivial, way. We can describe it succinctly as follows. Just for this proof, we will use the notation $[q]$ to stand for the set $\{0,1,\dots,q-1\}$ instead of the set $\{1,2,\dots,q\}$. 
	
	Let $n=m^s$. Then write the integers in $[m^r]$ in base $m$. Given any subset $A$ of $[r]$ of size $s$, and any integer $k\in [m^r]$, let $f_A(k)\in[m^s]$ be the number you get by restricting the base-$m$ representation of $k$ to the digits indexed by $A$. Now let $A_i=\{i, i+1,...,i+s-1\}$ mod $r$ for each $i\in[r]$, and define a sequence $T_0,T_1\dots,T_{m^r-1}$ of $r$-tuples by setting $T_k$ to be $(f_{A_1}(k),...,f_{A_r}(k))$ for each $k\in[m^r]$. 
	
	If $i<j$ then $f_{A_t}(i)<f_{A_t}(j)$ for any set $A_t$ that contains the highest coordinate that is less in the base $m$ representation of $i$ than in the base $m$ representation of $j$. There are $s$ such sets $A_t$, and so this sequence of $r$-tuples is $s$-increasing.
\end{proof}

Note that it was not important in the above construction that the sets $A_i$ were intervals mod $r$: all we needed was a collection of $r$ subsets of $[r]$, each of size $s$, such that every element of $[r]$ belonged to precisely $s$ of the sets. 

The result of Loh can also be easily generalized to improve the upper bound above by an $\exp(\Omega(\log^*n))$ factor.

It is now tempting to conjecture that the lower bound is sharp not just for 2-increasing sequences of triples, but more generally for $s$-increasing sequences of $r$-tuples. However, this turns out to be false. One way of seeing this is simply to note that the following example (discovered by a computer search, though it could probably have been found by hand) shows that $F_{4,2}(3)\ge 10>3^2$.

\begin{center}
(1,1,1,1)\\
(1,1,2,2)\\
(1,2,1,3)\\
(2,1,3,1)\\
(2,2,2,2)\\
(3,3,1,1)\\
(1,3,2,3)\\
(3,1,3,2)\\
(2,2,3,3)\\
(3,3,3,3)\\
\end{center}

But there is also a more conceptual argument, which makes it completely obvious that $n^{r/s}$ is not the right bound for all pairs $(r,s)$. If we fix $n$ to be 2, say, then for two random $r$-tuples $a$ and $b$, the expected number of coordinates for which $a_i<b_i$ is $r/4$, so by standard arguments the probability that $a$ is not $r/8$-less than $b$ is exponentially small in $r$. It follows easily that $F_{r,r/8}(2)$ is exponentially large in $r$, whereas if the $n^{r/s}$ bound were sharp, then $F_{r,r/8}(2)$ would be at most $2^8$.

These counterexamples weaken the case for believing that $F_{3,2}(n)\leq n^{3/2}$, and they suggest that giving an exact formula for $F_{r,s}(n)$ is unlikely to be possible for all triples $(r,s,n)$. They also tell us that any proof that $F_{3,2}(n)\leq n^{3/2}$ will have to have some aspect that cannot be generalized to all pairs $(r,s)$ -- indeed, not even to the pair $(4,2)$. 

\subsection{Our main results}

Our main result is the following theorem, which is presented in the next section. It provides a non-trivial power-type improvement to the upper bound for Problem \ref{PS1}.

\begin{theorem}\label{PBv2}
	There exists $\epsilon>0$ such that every 2-increasing sequence of triples taking values in $[n]$ has size at most $n^{2-\epsilon}$.
\end{theorem}

This is the first improvement over Loh's $n^2/\exp(\Omega(\log^*(n)))$ bound. Our proof makes essential use of the assumption that the sequence in question is 2-increasing and not just 2-comparable, so it does not yield an improvement for Problem \ref{PS2}. Also, the explicit $\epsilon$ we obtain is very small indeed, though as we shall explain later, if we had unlimited computer power then it could probably be improved substantially, though not to the point where it matches the lower bound.

Our second main result concerns the problem for 2-comparable sets of triples. We have not been able to improve on Loh's upper bound in this case, but, rather to our surprise, we found an example that beats the $n^{3/2}$ lower bound, which yields the following result.

\begin{theorem}\label{UB}
	For arbitrarily large $n$ there exist 2-comparable sets of triples of size at least $n^{1.546}.$
\end{theorem}

We shall describe the construction that proves this theorem in Section \ref{[32]}, before moving on to discuss a few interesting variants of the problem and connections to widely studied Tur\'an-type problems.

These two results suggest that the problems for 2-increasing sequences and 2-comparable sets of triples are fundamentally different, despite what the bounds for small examples suggest, though of course they do not actually prove that the exponents for the functions $F_{3,2}$ and $G_{3,2}$ are distinct.

In the final section we shall discuss the generalized problem for $[r,s]$-sequences. Our focus will switch from fixing $r$ and $s$ to fixing $n$ and the ratio $r/s$. This problem has some similarities with well-known results about unit vectors with upper bounds on their inner products, where the form of the bound depends strongly on whether the upper bound is positive, negative, or zero. We prove the following theorem, which shows a similar change in behaviour, for similar reasons, though our proofs are somewhat different, and the differences appear to be necessary.
\begin{theorem}
	Let $n\in\mathbb N$ and $\beta\in (0,1)$ be fixed. Then
	\begin{enumerate}[(i)]
		\item if $\beta<(1-1/n)/2$, then $G_{r, \beta r}(n)$ grows exponentially in $r$,
		\item if $\beta=(1-1/n)/2$, then  $G_{r, \beta r}(n)$ grows at least linearly in $r$, and
		\item if $\beta>(1-1/n)/2$, then  $G_{r, \beta r}(n)$ is bounded independently of $r$.
	\end{enumerate}
\end{theorem}

The significance of the number $(1-1/n)/2$ is that if $a$ and $b$ are random $r$-tuples taking values in $n$, then the expected proportion of coordinates $i$ for which $a_i<b_i$ is $r(1-1/n)/2$. (This quickly implies (i), as we have already observed in the case $n=2$ and $\beta=1/8$.)

\section{An upper bound for $(3,2)$-sequences} \label{2increasingcase}

In this section we shall prove our upper bound for $F(n)$. It may be of interest that this approach was only discovered after a significant amount of time considering a different, but more ``obvious" approach. The idea was to decompose sequences into smaller subsequences and use a combination of induction and Cauchy-Schwarz to prove the conjectured bound. Despite this method initially seeming promising, we did not manage to make it work. In Section~\ref{conc} we shall give a brief discussion of the obstacles that we discovered along the way. 

For the purposes of obtaining a convenient inductive hypothesis later, it will be useful to generalize Problem \ref{PS1} so that instead of taking the triples from $[n]^{3}$, we shall take them from a grid $[r]\times[s]\times [t]$, where the sides may have unequal lengths. The maximal length of a 2-increasing sequence now depends on the three parameters $r,s$ and $t$, and the trivial upper bound is $\min\{rs,rt,st\}$. Note that if we could ever find an example of a 2-increasing sequence of length greater than $(rst)^{1/2}$, then taking the product (in the sense described earlier) of this example and two further copies with the roles of the coordinates cycled round would give a 2-increasing sequence of length greater than $(rst)^{3/2}$ taking values in $[rst]$.

We now state our main result in a slightly generalized form.

\begin{theorem}\label{PB}
	There exists $\theta<2/3$ such that any 2-increasing sequence of triples from $[r]\times[s]\times[t]$ has size at most $(rst)^{\theta}$.
\end{theorem}

Note that if $r=s=t=n$, then the bound we obtain is $n^{3\theta}$. Thus, any improvement on $2/3$ for the exponent $\theta$ translates directly into an improvement on the exponent 2 for the problem as it was stated before. Unfortunately the improvement over $\frac23$ that we obtain is tiny. The main reason for this is that we need as a base case for an inductive argument an $n$ for which the trivial bound is beaten by a reasonable-sized constant. Finding such an $n$ by brute force is not computationally feasible, so we are forced instead to use Loh's upper bound (Theorem \ref{PS}). But then the $n$ in question is huge, so the exponent in the base case is only very slightly less than 2. So in a certain sense, the weakness in our argument is not a fundamental one.

However, as it stands, our argument still could not give a bound particularly close to the conjectured $(rst)^{1/2}$ even if we could use an arbitrarily large amount of computational power for the base case. The reasons for this will become clearer later, and we shall discuss this point further at the end of the section.

\subsection{Proof of Theorem \ref{PB}}
The proof will be by induction. Since the argument cannot hope to produce anything other than a $\theta$ very close to $\frac23$, we shall not put much effort into optimizing the details and shall aim instead for simplicity and clarity.

\subsubsection{Acyclic sets of triples}

Let us call a set $T$ of triples \emph{acyclic} if the restriction of the relation $<_2$ to $T$ contains no directed cycles. Note that a 2-comparable set $T$ of triples is in fact a 2-increasing sequence of triples if and only if it is acyclic, which is a useful observation because it allows us to study the problem for 2-increasing sequences as a problem about \emph{sets} of triples that avoid certain configurations. 

In the proof that follows, we shall use the acyclic property in a central way. In fact, we shall begin by considering the acyclic property alone -- that is, \emph{without} insisting on 2-comparability -- and obtaining the following upper bound for the size of an acyclic set of triples $T\subset[n]^2$. This demonstrates that the acyclic property is a strong condition to impose. 

\begin{lemma}\label{key}
	An acyclic subset of $[n]^3$ has size at most $6n^2$.
\end{lemma}

In order to achieve the $6n^2$ upper bound it in fact suffices only to ban directed 3-cycles. With this in mind, we shall deduce Lemma \ref{key} from the following slightly more technical statement. Define two elements $a,b$ of $[n]^2$ to be \emph{weakly 2-comparable} if $a_1\leq b_1$ and $a_2\leq b_2$ or if $a_1\geq b_1$ and $a_2\geq b_2$.

\begin{lemma}\label{preK}
	Let $A$ be a subset of $[n]^2$ that contains no three pairs $x$, $y$ and $z$ with $x<_2 y$ and $z$ not weakly 2-comparable to either $x$ or $y$. Then $|A|\le 4n-5$.
\end{lemma}

\begin{proof}
	Define a ``skew" ordering on $A$ by saying that $a\leq b$ if $a_1\geq b_1$ and $a_2\leq b_2$. Let $x_1,\dots,x_m$ be the minimal elements of $A$ in this order. Considering $A$ as a collection of points in the plane, these are the elements that have nothing below and to the right of them.
	
	For each minimal element $x_i$, let $X_i$ be the set of all points that are greater than $x_i$ in the second coordinate and smaller than $x_i$ in the first coordinate -- these are the points strictly above and strictly to the left of $x_i$. Also, let $\overline{X_i}$ be the set of points that are at least as big as $x_i$ in the second coordinate and at most as big in the first. Define the \emph{boundary} of $\bigcup_iX_i$ to be the set $\bigcup_i\overline{X_i}\setminus\bigcup_iX_i$. Then the following three conditions must be satisfied.
\begin{enumerate}
\item Every point in $A$ belongs to the set $\bigcup_i\overline{X_i}$.

\item If $y,z\in A\cap X_i$, then $y$ is not 2-less than $z$. In other words, $A\cap X_i$ is totally ordered by the skew ordering.
	
\item If $i\ne j$, then $A\cap X_i\cap X_j=\emptyset$.
\end{enumerate}
	
\noindent The first fact follows from the fact that $x_1,\dots,x_m$ are all the minimal elements. The second follows because if it were false then the points $x_i, y$ and $z$ would form a forbidden configuration, and the third fact follows because if $y\in A\cap X_i\cap X_j$, then the points $y$, $x_i$ and $x_j$ would form a forbidden configuration.
	
	It follows that the only points in $A$ belong either to the boundary of $\bigcup X_i$, which is a collection of points along a path that moves always either upwards or to the right, or to one of $m$ sets $B_i=X_i\setminus\bigcup X_j$.
	
	There are at most $2n-1$ points in any increasing path, and in this case $m$ of the elements are the $x_i$ themselves. As for the sets $B_i$, they are subgrids, and no two of them share a row or column. Moreover, for each $B_i$ we have that $A\cap B_i$ is totally ordered in the skew ordering. This last condition implies that if $B_i$ is a $u_i\times v_i$ subset of $[n]^2$, then $A\cap B_i$ has cardinality at most $u_i+v_i-1$. (To see this, observe that if we arrange the elements in an increasing sequence in the skew ordering, then as you move along the sequence, the first coordinate never increases, the second never decreases, and at least one of them always changes.) Moreover, the subgrids $B_i$ cannot intersect the boundary of $\cup X_i$ and so the sum of the dimensions is $\sum_i u_i+v_i\le 2n-2$.
	
	From this it follows that $\bigcup (A\cap B_i)$ has cardinality at most $2n-2-m$. This gives a bound of $2n-2-m+2n-1=4n-3-m$. If $m\ge 2$ we are done, and if $m=1$ then the above argument gives a bound of $4n-4$.
	
	Suppose our collection $A$ has size $4n-4$. Then it must be that there is a single skew-minimal element $x_1$ which must lie in the bottom right corner at the point $(n,1)$, since otherwise we obtain a saving in the length of the boundary path. Following the above argument, we see that the whole of the bottom row and rightmost column must be contained in our set (this is the boundary path). This tells us that the whole square $[3,n]\times[1,n-2]$ must be empty, else we combine with the points $(n-1,1)$ and $(n,2)$ to form a banned configuration. To obtain the required $4n-4$ points it follows that $A$ contains the bottom two rows and the rightmost two columns. But then the collection $\{(n-2,2),(n-1,3),(n,1)\}$ is contained in $A$ and is a banned configuration.
	
	Therefore $|A|\le 4n-5$.
	\end{proof}
	
	Note that the above lemma is sharp, because the set $$\{(a_i,b_i)|a_i\in\{n,n-1\} \text{ or } b_i\in\{1,2\}\}\setminus\{(n,1)\}$$ 
	satisfies our conditions and has size $4n-5$. (Of course, we do not need this level of precision, but one might as well give a sharp bound if one can.)

\begin{proof}[Proof of Lemma \ref{key}]
	Let $T$ be an acyclic set of triples. For each $(x,y)\in[n]^2$, throw away the triples $(x,y,z)\in T$ for which $z$ is largest and smallest, if any such triples exist. That throws away at most $2n^2$ triples. Let the resulting set of triples be $S$.
	
	Suppose that some $z$ is used at least $4n$ times as the third coordinate of a triple in $S$. Let $A$ be the set of points $(x,y)$ such that $(x,y,z)\in S$. Then by Lemma \ref{preK} we can find three points $a,b,c\in A$ with $a<_2 b$ and $c$ not weakly 2-comparable to either $a$ or $b$. 
	
	Now we split into two cases. Suppose first that $c_1>a_1$. Then $c_2<a_2<b_2$, so $c_1>b_1$ as well. Since $(a_1,a_2,z)\in S$ and $z$ is not the largest third coordinate for $(a_1,a_2)$, we can find a triple $(a_1,a_2,w)\in T$ with $w>z$. Similarly, we can find a triple $(b_1,b_2,v)\in T$ with $v<z$. These two triples, together with the triple $(c_1,c_2,z)$, form a 3-cycle since $(a_1,a_2,w)<_2(b_1,b_2,v)<_2(c_1,c_2,z)<_2(a_1,a_2,w)$. 
	
	If $c_1<a_1$, the proof is very similar. Therefore, we cannot find $4n$ triples in $S$ that share a third coordinate. It follows that the number of triples in $S$ is at most $4n^2$, so the number of triples in $T$ is at most $6n^2$ and we are done. 
		\end{proof}

We will now move on to providing the base cases that we need for the induction argument. 

\subsection{The base case}

First, we let $N$ be the minimal positive integer such that $(2(N+1)^3/N^3)^{2/3}\le 5/3$. We write $N$ for this constant, which will appear throughout the inductive step, for the sake of conciseness.

For our base case, we need to find a positive integer $k$ and a real number $\theta<2/3$ such that if $\min\{r,s,t\}\leq Nk$, then every 2-increasing subset of $[r]\times[s]\times[t]$ has size at most $(rst)^\theta$. We obtain this by combining Loh's result (Theorem \ref{PS} above) with some simple observations.

First, we choose an integer $k$ with the property that any 2-increasing sequence of triples in $[k]^3$ has length at most $\delta k^2$, where $20 \delta^{1/10} = k^{-\epsilon}$ and $\epsilon$ is some positive constant. The existence of such a $k$ and $\epsilon$ follows from Theorem \ref{PS}. 

Having chosen $k$ and $\epsilon$, let $\theta_1=(2-\epsilon)/3$. It will turn out that we need to take $\theta\ge \theta_1$ for our inductive hypothesis to work.

Once we have chosen our $k$, we need every 2-increasing sequence of triples from $[r]\times[s]\times[t]$  with $\min\{r,s,t\}\le Nk$ to have length at most $(rst)^\theta$. This places further strong constraints on how small we are able to take $\theta$.

Without loss of generality, $r\leq s\leq t$. Then in order to ensure that the condition is satisfied, we first note that whenever $r,s$ and $t$ are not all equal the trivial bound $rs$ is equal to $(rst)^{\tau}$ for some $\tau(r,s,t)=\log(rs)/\log(rst)<2/3$. The expression on the left-hand side decreases as $t$ increases and increases as $s$ increases, so it is maximized, for fixed $r$, when $s=t=r+1$ (using our assumption that $r\leq s\leq t$ and that $r\ne t$). Now allowing $r$ to vary between 1 and $Nk$ we find that $\tau(r,s,t)$ is maximized when $r=Nk, s=t=Nk+1$, when it takes the value $\log(Nk(Nk+1))/\log(Nk(Nk+1)^2)$. Let us call this maximum $\theta_2$. We will need $\theta$ to be at least $\theta_2$.

It remains to deal with the cases in which $r=s=t\le k$. For this we need a simple lemma.

\begin{lemma} \label{easybound}
A 2-comparable set $T$ of triples in $[r]^3$ has size at most $t(r)$, where $t(r)=3r^2/4$ if $r$ is even and  $t(r)=3r^2/4+r/2+3/4$ if $r$ is odd.
\end{lemma}

\begin{proof}
Let $A$ be the set of all $(x,y)\in[r]^2$ such that $(x,y,z)\in T$ for some $z$. If such a $z$ exists, it is unique, by the 2-comparability condition, so let us call it $f(x,y)$. 

Suppose that $(x_1,y_1), (x_2,y_2)\in A$ and $\max\{x_1,y_1\}=\max\{x_2,y_2\}$. Then $f(x_1,y_1)$ and $f(x_2,y_2)$ are distinct, since either $x_1=x_2$, $y_1=y_2$, or $x_1$ and $x_2$ are not ordered in the same way as $y_1$ and $y_2$. Here again we are using 2-comparability.

For $i=1,2,\dots,r$, let $A_i=\{(x,y)\in A:\max\{x,y\}=i\}$. Then trivially $|A_i|\leq 2i-1$, and the argument just given shows also that $|A_i|\leq r$. It follows that $|A|\leq\sum_{i=1}^{\lfloor r/2\rfloor}(2i-1)+r\lceil r/2\rceil$. If $r$ is even, this equals $(r/2)^2+r^2/2=3r^2/4$. If $r$ is odd, then it is $((r-1)/2)^2+(r+1)^2/2$, which equals the bound stated.
\end{proof}

Actually all we really need is that $T$ has size strictly less than $r^2$ when $r>1$: the above result improves our eventual bound, but not in an interesting way.

For each $r>1$, define $\tau(r)$ so that $r^{3\tau(r)}=t(r)$: that is, $\tau(r)=\log(t(r))/3\log r$. Let $\theta_3=\max\{\tau(r):r\leq Nk\}$. We shall also need the inequality $\theta\geq\theta_3$ for our proof to work.

We now fix $\theta=\max(\theta_1,\theta_2,\theta_3)$ and proceed with the inductive step of the argument.

\subsection{The inductive step}

Let $T$ be a transitive 2-comparable subset of $[r]\times [s]\times [t]$. We form a quotient set $T'\subset[k]^3$ by dividing each dimension into $k$ intervals as equally as possible. That is, if our divisions into intervals are $[r]=R_1\cup\dots\cup R_k$, $[s]=S_1\cup\dots\cup S_k$ and $[t]=T_1\cup\dots\cup T_k$, then $T'=\{(h,i,j):T\cap(R_h\times S_i\times T_j)\ne\emptyset\}$. We will assume that $\min\{r,s,t\}>Nk$, since otherwise we have one of our base cases and therefore the required estimate $|T|\leq(rst)^\theta$. 

This quotient operation does not preserve 2-comparability, but, crucially, it \emph{does} preserve the acyclic property. This follows simply from the fact that if we have a directed cycle of quotient triples then by taking a representative triple $t\in T$ from each quotient triple we get a directed cycle in $T$. It is for this reason that it is so useful to us that the acyclic property alone has strong consequences.

At this point, we could naively bound the number of triples in $T$ by applying our inductive hypothesis to bound the number of triples contained in each quotient triple, and multiplying by our upper bound on the size $|T'|$ of the quotient set. This gives us a bound of
\[(6k^2)\left(\frac{(N+1)^3rst}{N^3k^3}\right)^\theta.\]
Unfortunately this is larger than $(rst)^\theta$ when $\theta<2/3$, and so this is not quite powerful enough to complete the induction.

However, we can improve on this by grouping the quotient triples into collections for which we may obtain an improved estimate using our inductive hypothesis. For this, we use the following definition and lemma.

\begin{definition}
	Let $H$ be a collection of integer triples entirely contained in one of the planes $(x,*,*)$, $(*,y,*)$ or $(*,*,z)$. Suppose that when we project $H$ onto the two free coordinates (obtaining a collection $H_p$ of integer pairs) we have no two elements of $H_p$ that are 2-comparable as pairs, in the obvious sense. Then we say that $H$ is a \emph{collapsible} collection of triples.
\end{definition}

It turns out that we can apply our inductive hypothesis to bound more efficiently the number of triples from $T$ in a collection of quotient labels when the collection is collapsible.

\begin{lemma}\label{collapsible}
	Let $H$ be a collapsible collection of triples in the quotient set. Then the total number of triples from $T$ contained in the quotient triples of $H$ is at most
	\[\left(\frac{2(N+1)^3|H|rst}{N^3k^3}\right)^\theta.\]
\end{lemma}
\begin{proof}
	 Let us assume (without loss of generality) that the triples in $H$ agree in their third coordinate, and let this coordinate be $z$. So the triples can be written in a sequence as $(u_1,v_1,z),\dots,(u_m,v_m,z)$ with $u_1\geq\dots\geq u_m$ and $v_1\leq\dots\leq v_m$. 
	
	Let us partition $H$ into two sets $U$ and $V$, where $U$ is the set of $(u_i,v_i,z)$ such that $u_i<u_{i+1}$ and $V=H\setminus U$. Then if $i<j$ and $(u_i,v_i,z),(u_j,v_j,z)\in U$, we have that $u_i<u_j$. Also, if $i<j$ and $(u_i,v_i,z),(u_j,v_j,z)\in V$, then $v_i<v_j$, since if $v_i$ were to equal $v_j$ then $v_i=v_{i+1}$, which implies that $u_i>u_{i+1}$ and therefore that $(u_i,v_i,z)\in U$. Thus, we have partitioned $H$ into two sets, in one of which the $u_i$ strictly increase, and in the other of which the $v_i$ strictly increase.
	
	Now let us partition $U$ further into sets $U_i$, according to the value of the second coordinate. The main fact that enables us to get a good bound is that if $i\ne j$ and $q$ is the quotient map, then no point in $q^{-1}(U_{i})$ can share a third coordinate with a point in $q^{-1}(U_{j})$. That is because if $i<j$, then points in $q^{-1}(U_{i})$ have a higher first coordinate and a lower second coordinate than points in $q^{-1}(U_{j})$. 
	
	Let us suppose then that $|I_{i}|=a_i$ and that $c_i$ different third coordinates occur in $q^{-1}(U_{i})$. Then $\sum a_i=|U|$ and $\sum c_i\le \lceil t/k\rceil\le (N+1)t/Nk$. Also, by our inductive hypothesis, the number of points in $q^{-1}(U_{i})$ is at most $((N+1)^2a_irsc_i/N^2k^2)^\theta$, since they live in a Cartesian product of three sets that have sizes at most $(N+1)a_ir/Nk$, $(N+1)s/Nk$, and $c_i$. Summing, over $i$, we find that 
	\[|T\cap q^{-1}(U)|\leq\sum_i\bigg(\frac{(N+1)^2a_ic_irs}{N^2k^2}\bigg)^\theta.\]
	Similarly, we can partition $V$ into sets $V_{i}$ with $|V_{i}|=b_i$ and at most $c_i'$ different third coordinates occurring in $q^{-1}(V_{i})$, then $\sum b_i=|V|$ and $\sum c_i'\leq (N+1)t/Nk$, and we have the bound
	\[|T\cap q^{-1}(V)|\leq\sum_i\bigg(\frac{(N+1)^2b_ic_i'rs}{N^2k^2}\bigg)^\theta.\]
	
	Now
	\[\sum(x_jy_j)^{\theta}\le\big(\sum x_j\big)^\theta\big(\sum y_j^{\theta/(1-\theta)}\big)^{1-\theta}\le\big(\sum x_j\big)^\theta\big(\sum y_j\big)^\theta\]
	by H\"older's inequality, the monotonicity of $l_p$ norms, and the fact that $\theta\geq 1-\theta$. Applying this to the sum of the above two expressions and using our bounds for $\sum a_j$ and $\sum b_h$, $\sum c_j$ and $\sum c_h'$, we deduce that
	\[|T\cap q^{-1}(H)|\le\bigg(\frac{2(N+1)^3|H|rst}{N^3k^3}\bigg)^\theta.\]
\end{proof}

Now the key idea is to partition the quotient set into two parts, the first of which is a union of large collapsible collections and the second of which is a genuine 2-increasing sequence. The contribution to the size of $T$ from the first part will be controlled by using the collapsibility, while the second part will be controlled by the bound on the length of a 2-increasing sequence in $[k]^3$ obtained in the base case.

This splitting is achieved using the following lemma.

\begin{lemma}\label{split}
	Suppose that $S$ is a collection of triples containing no collapsible collection of size $C$. Then $S$ contains a 2-comparable subset of size at least $C^{-3}|S|$.
\end{lemma}
\begin{proof}
	For the plane $P_x=(x,*,*)$, let $S_x=S\cap P_x$. Clearly the triples in the set $S_x$ are partially ordered by $<_2$, and the antichains in this set are precisely the collapsible collections.
	
	Since $S$ has no collapsible collection of size larger than $C$, we have that $S_x$ has no antichain of length greater than $C$ and therefore (by Mirsky's Theorem) it must have a chain $S_x'$ of length at least $C^{-1}|S_x|$. 
	
	Let $S_1$ be the subset $\cup_x S_x'$. We see that $|S_1|\ge C^{-1}|S|$.
	
	Now we do the same with the $y$-coordinate, obtaining a subset $S_2$, and then again with the $z$-coordinate, obtaining a subset $S_3$. We have that $|S_3|\ge C^{-3}|S|$, and for any subset of $S_3$ obtained by fixing a coordinate the elements of this subset are totally ordered by $<_2$.
	
	This means that $S_3$ is 2-comparable, since for two triples to fail to be 2-comparable they must share a coordinate and thus must both lie in one of the planes that we have treated above. Since restricting $S_3$ to this plane gives a subset totally ordered by $<_2$, the triples must be 2-comparable.
\end{proof}

Let $C$ be a fixed constant, which we shall specify later. We may repeatedly extract collapsible collections of size $C$ from the quotient set $T'$ until we are left with a set $S$ at which point the extraction fails. When that happens, Lemma~\ref{split} implies that $S$ must have a 2-comparable subset $S'$ of size $C^{-3}|S|$. 

However, since $S'\subset T'$ and $T'$ is acyclic, $S'$ is also acyclic, which implies that it corresponds to a 2-increasing sequence (since for 2-comparable sets the acyclic property implies transitivity of the relation $<_2$). Since $T'$ contains no 2-increasing sequence of length $\delta k^2$ by our base case, we have that $C^{-3}|S|\le \delta k^2$.

Now we may use Lemma~\ref{collapsible} to bound the number of triples in $T$. We have split the quotient set $T'$ into a set $S$ of size at most $C^3\delta k^2$, and the rest of $T'$ which partitions into collapsible collections of size $C$. We therefore find that
\[|T|\le \frac{|T'|-|S|}{C}\left(\frac{2(N+1)^3Crst}{N^3k^3}\right)^\theta+|S|\left(\frac{2(N+1)^3rst}{N^3k^3}\right)^\theta\]
\[\le \left(\frac{6k^2}{C}\left(\frac{2(N+1)^3C}{N^3k^3}\right)^\theta+C^3\delta k^2\left(\frac{2(N+1)^3}{N^3k^3}\right)^\theta\right)(rst)^\theta.\]
Taking $C$ to be such that $C^3\delta=6C^{2/3}/C=A$ we get $C=6^{3/10}\delta^{-3/10}$ and $A=6^{9/10}\delta^{1/10}$. Therefore 
\[|T|\le \left(2.6^{9/10}\delta^{1/10}k^2\left(\frac{2(N+1)^3}{N^3k^3}\right)^\theta\right)(rst)^\theta\]
which, by our choice of $N$, is
\[\le \left(20\delta^{1/10}k^{2-3\theta}\right)(rst)^\theta.\]
But our choice of $k$ from the base case gives us that
\[20\delta^{1/10}k^{2-3\theta}\le k^{-\epsilon}k^{2-3\theta}=k^{2-3\theta-\epsilon}\]
and 
\[2-3\theta-\epsilon<0\]
by our choice of $\theta$ so the induction follows and the proof of Theorem \ref{PB} is complete.

\subsection{Remarks}
\subsubsection{Size of $\theta$}
Here we shall give a very brief examination of the size of the $\theta$ that emerges from the argument. It is not worth being too careful here, as we have made little effort to tighten up the argument and because the use of Proposition~\ref{PS} means that the difference $\frac23-\theta$ is unavoidably extremely small.

First of all, it is important to get an explicit version of Proposition~\ref{PS} that gives us a constant to replace the $\Omega$ notation. For this we can use the best known bound for the triangle removal lemma, due to Fox~\cite{Fox}, and we obtain a quantitative version of Proposition~\ref{PS}, namely that
$$F(n)\le n^2/\exp(\log^*(n)/405).$$
It is also easy to check that in the base case $\theta_2\ge \theta_3$ so $\theta_3$ is of no concern.

In order to get $\epsilon>0$ in the expression $20\delta^{1/10}\le k^{-\epsilon}$, using $\delta=\exp(-\log^*(k)/405)$ as is allowed by the above, we need 
$$\exp(\log^*(k)/405)>20^{10}$$
and so we will need $k\ge T(405\log(20^{10})$, where $T$ is the tower function. Note that $405\log(20^{10})<12133$. We have that
\[\theta_2=\frac{\log(Nk)+\log(Nk+1)}{\log(Nk)+2\log(NK+1)}\]
\[=2/3-\frac{1}{9Nk\log(Nk)}+\mathcal{O}(1/(Nk)^2)\]
and since $k$ is huge this gives us $\theta_2\approx 2/3-\frac{1}{9Nk\log(Nk)}$. Certainly, if we take $k=T(12133)$, say, then we have $\theta_2<2/3-\frac{1}{T(12133)}$.

All that remains is $\theta_1$, which is given by $(2-\epsilon)/3$ where $k^{-\epsilon}>20\delta^{1/10}=20\exp(-\log^*k/4050)$. If we take $k=T(12133)$ then we have
$$\epsilon=\log(\exp(12133/4050)/20)/\log(T(12133))>1/T(12133)$$
and so certainly $\theta_1<2/3-\frac{1}{T(12133)}$ also.

Putting this together, we are able to choose $\theta=2/3-\frac{1}{T(12133)}$. With more effort to optimize the proof, the $T(12133)$ might be able to be brought down somewhat but a significant change to the base case is required to avoid the tower function.
\subsubsection{Limitations and Scope for Improvements}
The key to the argument that we have just given is that we may use the acyclic property on its own to bring the size of $H$ down from order $k^3$ to order $k^2$. Once we have realized this fact, it is fairly clear that we should be able to make a power-type improvement over the trivial $n^2$ bound on $T$ by partitioning the quotient structure, which can be controlled by using the acyclic property, into collections for which we can apply the inductive hypothesis efficiently.

However, there is a fundamental slackness in the argument as described above, since even if we could take $\delta=k^{-1/2}$ in the base case (the best we could hope for), we would end up with $\epsilon\approx 1/20$ and a rather tiny improvement to the upper bound. Therefore, even if we had enough computational power to verify for any finite $k$ the conjectured bound of $(rst)^{1/2}$ for the maximal length of 2-increasing sequences from $[r]\times[s]\times[t]$ with $r\le k$, so that we would could get $\epsilon$ as close as we like to 1/2 in the base case, we would only be able to obtain a bound for $\theta$ that was arbitrarily close to $2/3-1/20$ rather than to the $1/2$ that we would expect.

One way that we could hope to improve this is to gain a better understanding of the structure of acyclic sets. In the current argument we observe that the quotient structure is acyclic, which limits the number of labels to $\mathcal{O}(k^2)$, but then we fall back on rather primitive methods to decompose it into collapsible subsets. Indeed, collapsible subsets are not the only ones for which we can obtain a more efficient application of the inductive hypothesis. If we were always able to decompose acyclic sets of triples into a wider class of subsets that allow for efficient induction we could hope to improve the argument substantially. It seems very likely, therefore, that one can do better than this, especially since the structure of acyclic sets with almost the maximum size seem to be quite restricted. 

\section{A lower-bound for $[3,2]$-sequences}\label{[32]}
In this section we shall describe a construction that beats the $n^{3/2}$ lower bound. We will then discuss the upper bound, for which any improvement over the result of Loh has proved elusive.

\subsection{A reformulation using labels in grids}

In this section we will be presenting various examples of 2-comparable sets of triples. If they are presented just as lists, then it is somewhat tedious to check that they are 2-comparable. However, there is a simple reformulation that is much more convenient for the purposes of looking at and understanding small examples of $[3,2]$-sequences, and also $(3,2)$-sequences. We briefly describe it here.

Given a 2-increasing sequence $T$ of triples from $[r]\times[s]\times[t]$, we define the \emph{grid representation} of $T$ by considering each triple as a labelled point in the grid $[r]\times[s]$. That is, we think of the triple $(a,b,c)$ as the point $(a,b)$ labelled with $c$. Thus the whole sequence $T$ corresponds to a labelling of some of the points of an $r\times s$ grid with labels from $[t]$.

As an example, the grid representation of the set 
\[T=(1,1,1),(1,2,2),(2,1,3),(2,2,4),(3,3,1),(3,4,2),(4,3,3),(4,4,4)\]
is
\[\young({~~24,~~13,24~~,13~~})\]

Of course there is no particular reason to consider the third coordinate to be the label coordinate, and it is sometimes instructive to look at the same example in three different ways.

Now let us think about the restrictions imposed on labelled subsets of the grid if they are grid formulations of 2-increasing sequences of triples. 

We begin by considering what follows from the 2-comparability condition. Note that if two triples do not share a coordinate, then they are automatically 2-comparable, so the condition is equivalent to saying that if $a$ and $b$ are two triples that share one coordinate, then either $a$ is less than $b$ in both the other coordinates, or $a$ is greater than $b$ in both the other coordinates. It follows from this that in the grid representation, if two points are in the same row, then the point to the right has a higher label than the point to the left, and if two points are in the same column, then the higher point has a higher label than the lower point. To put this more concisely, labels strictly increase as you go along a row or up a column. If it is the label coordinate that is fixed, then the condition states that the points with a given label must form a sequence that moves up and to the right, or in other words a 2-increasing sequence of pairs. That is, if $(x_1,x_2)$ and $(y_1,y_2)$ have the same label, then either $x_1<y_1$ and $x_2<y_2$ or $y_1<x_1$ and $y_2<x_2$. Equivalently (given that the same label cannot occur twice in a row or column), if $x_1<y_1$ but $x_2>y_2$, then $(x_1,x_2)$ cannot have the same label as $(y_1,y_2)$. 

The additional constraint in the 2-increasing case is that $T$ must be acyclic, which, if $T$ is 2-comparable, is equivalent to saying that the relation $<_2$ is transitive when it is restricted to $T$. In the grid representation, a collection of triples that violates transitivity corresponds to having cell $(a,b)$ filled with label $c$ and cell $(a',b')$ filled with label $c'$, where $a'>a,b'>b$ and $c'<c$, and having a third cell $(a'',b'')$ with label $c''$ and $c'<c''<c$ where either $a''>a'$ and $b''<b$ or $a''<a$ and $b''>b$. 

This configuration is much easier to express pictorially. Given two cells, with the one with smaller label $c$ above and to the right of the other with larger label $c'$, we define two regions of the grid. The first is the region above and to the left of both cells, and the second is the region below and to the right of both cells. To get a configuration that violates transitivity we simply place a label between $c$ and $c'$ in one of these regions. For an example, see Figure \ref{Fig1}.
\begin{figure}
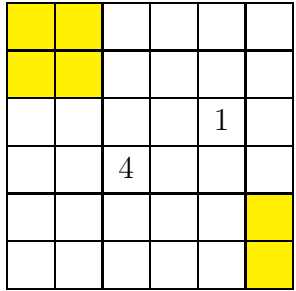

	\[
	\begin{ytableau}
	*(yellow) & *(yellow) &&&& \\
	*(yellow) & *(yellow) &&&& \\
	& &&&$1$& \\
	&  &$4$&&& \\
	&  &&&& *(yellow)\\
	&  &&&& *(yellow)\\
	\end{ytableau}
	\]
	\caption{An example of the two regions described previously, highlighted in yellow. If the label 2 or 3 is placed within one of these regions, we get an intransitivity.}
	\label{Fig1}
\end{figure}

Thus, grid representations of 2-increasing sequences are characterized by the three properties below, and grid representations of 2-comparable sequences are characterized by the first two properties.
\begin{enumerate}
	\item It increases along rows and up columns.
	\item The set of points with any given label forms a 2-increasing sequence.
	\item It must not contain a transitivity-breaking configuration of the kind just described.
\end{enumerate}

\subsection{A continuous generalization}

Here we will give a natural continuous generalization of the $[3,2]$ problem (and also the $(3,2)$ problem), which extends the grid formulation discussed in the previous section. We use the word ``cuboid" to mean an axis-parallel cuboid.

\begin{definition}
Let $I$ and $J$ be two real intervals. Say that $I<J$ if $x<y$ for every $x\in I$ and $y\in J$. If $I_1,I_2,I_3$ and $J_1,J_2,J_3$ are real intervals, then $I_1\times I_2\times I_3<_2J_1\times J_2\times J_3$ if $I_h<J_h$ for at least two values of $h$. If $C$ and $C'$ are two cuboids, then they are \emph{2-comparable} if $C<_2C'$ or $C'<_2C$. A sequence of cuboids $C_i\subset\mathbb R^3$ is \emph{2-increasing} if $C_i<_2C_j$ whenever $i<j$. It is \emph{2-comparable} if any two distinct $C_i$ are 2-comparable.
\end{definition}

Given a set of triples in $[r]\times[s]\times[t]$, we can convert it into as a collection of open unit cubes in the cuboid $[0,r]\times [0,s]\times[0,t]$ (where the triple $(a,b,c)$ corresponds to the unit cube with corner $(a,b,c)$ furthest from the origin). The resulting collection of cubes is 2-increasing/2-comparable if and only if the set or triples is 2-increasing/2-comparable.

This leads to the following generalization of the discrete question.

\begin{problem}\label{CV}
	Let ${B}=\{B_i\}$ be a set of disjoint open cuboids lying in $[0,1]^3$. Define $\|B\|_\alpha$ by the formula
	\[\|B\|_\alpha=\big(\sum_i |B_i|^{\alpha}\big)^{1/\alpha}.\]
	Let $\theta$ be the supremum over all $\alpha$ such that there exists a finite, 2-comparable collection $B$ of at least two cuboids with
	\[\|B\|_\alpha\geq 1.\]
	What is the value of $\theta$?
\end{problem}

Observe that if $|B|=1$ then we can take $B$ to consist of the whole unit cube and then $\|B\|_\alpha=1$ for all $\alpha$, so we exclude this case. If $|B|>1$ then for $\alpha>1$ we have $\|B\|_\alpha<\|B\|_1< 1$ since we cannot hope for the $B_i$ to cover the whole of the encompassing cube. This tells us that $\theta$ exists and is at most 1. 

Taking 
\[B_1=(0,1/2)\times (0,1/2)\times (0,1)\]
and
$$B_2=(1/2,1)\times(1/2,1)\times(0,1)$$
and setting $B=\{B_1,B_2\}$ we have that $\|B\|_{1/2}=1$, so $\theta\ge1/2$. 

We now show that this continuous generalization is, in a suitable sense, equivalent to the discrete problem.

\begin{lemma}\label{equiv}
	Let $\theta$ be such that there exists a finite, 2-comparable/2-increasing collection $B$ of at least two cuboids in $[0,1]^3$ with
	\[\|B\|_\theta=1.\]
	Then for any $\epsilon>0$ there exist $n$ and a finite, 2-comparable/2-increasing collection $T$ of integer tuples, each lying in $[n]^3$, with
	\[|T|\ge n^{3\theta-\epsilon}.\]
	The converse also holds, in the sense that given a collection $T$ with $|T|= n^{3\theta}$ we get a collection $B$ with $\|B\|_\theta=1$.
\end{lemma}

\begin{proof}
	The converse is easy, since, as already remarked, we can view the collection $T$ as a collection of unit cubes inside $[0,n]^3$, which we can then scale down by a factor of $n$. This gives a collection $B$ of at least two $1/n\times1/n\times1/n$ cuboids, and 
	\[\sum_{B_i}\left(\frac{1}{n^3}\right)^\theta=\frac{|T|}{n^{3\theta}}=1.\]
	
	The other implication is a little more subtle. What we would like to do is take the collection $B$ living inside $[0,1]^3$ and discretize it. To begin with, we would take a fine grid and take all the points in it that live inside $\bigcup B$. Although this does not give us a 2-comparable/2-increasing set, it gives us a set that splits up nicely into a disjoint union of subgrids. We could then hope to take 2-comparable/2-increasing subsets of these subgrids that are as large as possible and put them together. However, this approach runs into difficulties, because the subgrids could be of very different sizes and shapes, which makes it unclear that we can fit long 2-increasing/2-comparable subsets inside all of them simultaneously. (Recall, for instance, the trivial upper bound of $\min\{rs,rt,st\}$, which, if $r,s,t$ are sufficiently unbalanced, will be less than $(rst)^{1/2}$.)
	
	So first we shall ``treat" the collection $B$ so that the cuboids are all of comparable dimensions. This is done as follows.
	
	For any collection of cuboids $B$ we define a sequence $B^1,B^2,\dots$ by setting $B^1=B$ and defining $B^k$ by replacing each $B_i\in B^{k-1}$ by a suitably scaled copy of $B$. Note that we have $|B^k|=|B|^k$ and also 
	\[\sum_{B_i\in B^k}|B_i|^\theta=\sum_{B_i\in B^{k-1}}|B_i|^\theta\big(\sum_{B_j\in B}|B_j|^\theta\big)=\sum_{B_i\in B^{k-1}}|B_i|^\theta.\]
Therefore, by induction we have 
	\[ \sum_{B_i\in B^k}|B_i|^\theta=1\]
	for all $k$.
	
	Now suppose we have a collection $B$ such that $\|B\|_\theta=1$. First we choose a positive integer $m$ and perturb the cuboids in $B$ so that their sidelengths are all multiples of $m^{-1}$. For any $\epsilon>0$ we can choose $m$ and the perturbation in such a way that the peturbed collection $B'$ has $\|B'\|_{\theta-\epsilon}\ge 1$.
	
	Let us fix our $\epsilon>0$ and our choice of $m$. Then let $p_1,\dots,p_r$ be the primes less than or equal to $m$ and define the \emph{sidelength vector} of a cuboid in $B'$ to be the vector $(a_1,\dots,a_r,b_1,\dots,b_r,c_1,\dots,c_r)$, where the sidelengths of the cuboid are $m^{-1}p_1^{a_1}\dots p_r^{a_r}$, $m^{-1}p_1^{b_1}\dots p_r^{b_r}$ and $m^{-1}p_1^{c_1}\dots p_r^{c_r}$.
	
	We extend this definition of a sidelength vector to the cuboids in $(B')^k$ by assigning to a cuboid in $(B')^k$ the vector $(a_1,\dots,a_r,b_1,\dots,b_r,c_1,\dots,c_r)$ where the dimensions of the cuboid are $m^{-k}p_1^{a_1}\dots p_r^{a_r}$, $m^{-k}p_1^{b_1}\dots p_r^{b_r}$ and $m^{-k}p_1^{c_1}\dots p_r^{c_r}$. Having done this, we see that the sidelength vectors of cuboids in $B_p^k$ are just sums of $k$ of the sidelength vectors of cuboids in $B'$.
	
	The total number of sidelength vectors for $(B')^k$ is the size of the $k$-fold iterated sumset of the set of sidelength vectors for $B'$, and these all live in the box $[km]^{3r}$ so their number grows polynomially with $k$. Fix $k$ large, and let $v$ be the sidelength vector such that the sum of all $|B_i|^{\theta-\epsilon}$ such that $B_i\in(B')^k$ has sidelength vector $v$ is maximized.
	
	Let $C=\{C_i\}$ be the subcollection of $(B')^k$ consisting of the cuboids with sidelength vector $v$. Then
	\[\sum_i|C_i|^{\theta-\epsilon}\ge (km)^{-3r}.\]
	All the $C_i$ have the same sidelengths: let these be $d_1$, $d_2$ and $d_3$. Now subdivide each of the three sides of the unit cube into intervals of equal lengths $a, b$ and $c$, with $d_1/4\le a\le d_1/2$, $d_2/4\le b\le d_2/2$ and $d_3/4\le c\le d_3/2$. Let $D$ be a collection of cuboids obtained by selecting, for each $C_i$, precisely one of the $a\times b\times c$ cuboids that is entirely contained within $C_i$: by our choice of $a,b,c$ such a cuboid must exist.
	
	Since $\sum_i|C_i|^{\theta-\epsilon}\ge (km)^{-3r}$, we have that
	\[|C|(d_1d_2d_3)^{\theta-\epsilon}\ge (km)^{-3r},\]
from which it follows that
	\[|D|(abc)^{\theta-\epsilon}\ge (km)^{-3r}/64.\]
	
	We may now obtain a discrete sequence of tuples by scaling up the collection of cuboids $D$ by a factor of $A=1/a$, $B=1/b$ and $C=1/c$ in the three dimensions (note that $A,B,C\in \mathbb{Z})$ so that the cuboids become unit cubes. Let $T$ be the collection of integer tuples that we get by taking the furthest point in (the closure of) each cube from the origin. Then $T$ is a collection of integer tuples lying in $[A]\times[B]\times[C]$, and it is 2-increasing/2-comparable if $B$ is.
	
	We now observe that $ABC$ is exponentially large in $k$. This follows provided that we can show that at least one of $a$, $b$ or $c$ is exponentially small. But to any sidelength vector $v$ we may associate a sequence $S$ of $k$ cuboids from $B'$ such that the sidelengths defined by $v$ are the products of the corresponding sidelengths from $S$. Since $B'$ consists of more than one cuboid, the $1\times 1\times 1$ cuboid is not present, and consequently at least one third of the sidelengths in $S$ are not length 1. Letting $h<1$ be the largest non-unit sidelength of any cuboid in $B'$, we deduce that at least one of $a$, $b$ or $c$ is at most $h^{k/3}$.
	
Since $ABC$ is exponentially large in $k$, by taking $k$ sufficiently large we can ensure that
	\[|T|\ge(ABC)^{\theta-2\epsilon}.\]
	
	All that is now required is to build a collection of tuples from $T$ that live inside a set $[n]^3$ rather than $[A]\times[B]\times[C]$. We saw how to do this at the beginning of Section \ref{2increasingcase}. We let $\phi$ be the map that cycles the coordinates of each tuple round by one place, so $\phi(\{(a,b,c)\})=\{(c,a,b)\}$. Define $T_1=T$, $T_2=\phi(T)$ and $T_3=\phi^2(T)$. Then, using the definition of a product of two sequences given in the proof of Lemma~\ref{prod}, we can take the sequence $S=T_1\otimes T_2\otimes T_3$. $S$ is a set of integer tuples each lying in $[n]^3$ where $n=ABC$, and $|S|=n^{3\theta-6\epsilon}$, and it is 2-increasing/2-comparable if $T$ is.
\end{proof}

This lemma allows us to consider continuous constructions in our search for long 2-comparable sequences. 
It turns out, as we shall demonstrate in Section \ref{long2cs}, that this is quite useful.

There is a clear resemblance between the definition of $\theta$ above and the definition of Hausdorff dimension. It seems almost certain that the correct exponent in the discrete problems is equal to the maximal Hausdorff dimension of a subset of $[0,1]^3$ that is 2-increasing/2-comparable, but we have not attempted to prove this. 

One reason the continuous problem helps is that it allows us to use variational arguments. The next lemma illustrates this. Although it is not strictly necessary for our purposes (we shall make use of it, but will then prove a stronger result without using it), it may be important in future developments. That is because, as we shall see in Section~\ref{conc}, to prove an upper bound of $n^{3/2}$ for the 2-increasing problem, it appears to be necessary to use extremality, and this lemma is almost the only way we have found of doing that.

\begin{lemma} \label{extremal}
Let $B=\{B_i\}$ be a finite collection of disjoint open cuboids lying in $[0,1]^3$ and let $\alpha>0$. For each $i$, let $B_i=X_i\times Y_i\times Z_i$ and let $x_i=|X_i|, y_i=|Y_i|$ and $z_i=|Z_i|$. Given any $t\in[0,1]$, define $f(t)$ to be $\sum_{i:t\in X_i} x_i^{\alpha-1}(y_iz_i)^\alpha$. Then either $f$ is constant for almost every $t$ or there is a continuous piecewise linear bijection $\phi:[0,1]\to[0,1]$ such that if we set $C_i=\{(\phi(x),y,z):(x,y,z)\in B_i\}$ for each $i$ and $C=\{C_i\}$, then $\|C\|_\alpha>\|B\|_\alpha$.
\end{lemma}

\begin{proof}
If $f$ is not constant almost everywhere, then we can find $t$ and $u$ such that neither $t$ nor $u$ is the end point of any of the intervals $X_i$, and $f(t)\ne f(u)$.

Now choose small intervals $I$ and $J$ about $t$ and $u$ that do not contain the end points of any of the $X_i$ and choose a piecewise linear bijection $\phi:[0,1]\to [0,1]$ that has gradient 1 outside $I\cup J$, increases the length of $I$ by $\delta$, and decreases the length of $J$ by $\delta$. Then $|\phi(X_i)|=|X_i|$ for every $i$ such that $X_i$ contains both $t$ and $u$ or neither $t$ nor $u$. If it contains just $t$ then $|\phi(X_i)|=|X_i|+\delta$ and if it contains just $u$ then $|\phi(X_i)|=|X_i|-\delta$.

Now let us think about how the sum $\sum_i(x_iy_iz_i)^\alpha$ changes when we expand and contract the intervals $X_i$ in this way. The effect of increasing $x_i$ by $\delta$ is to increase the sum by $\alpha x_i^{\alpha-1}y_iz_i\delta+o(\delta)$ and the effect of decreasing it by $\delta$ is to decrease the sum by that amount. Therefore, 
\[\|C\|_\alpha^\alpha-\|B\|_\alpha^\alpha = \alpha\delta(f(t)-f(u))+o(\delta).\]
Since $f(t)\ne f(u)$, we can choose $\delta$ (possibly negative) such that the right-hand side is positive, and the result is proved.
\end{proof}

Note that the map $(x,y,z)\mapsto(\phi(x),y,z)$ preserves all the order relations we are interested in, so if $B$ is 2-increasing or 2-comparable then so is $C$. So the lemma implies that if we have an extremal example in the continuous case, then all its cross sections (apart from those that intersect the boundaries of the cuboids) are of the same ``size'', as measured by the function $f$. Note too that if all the cuboids have the same size and shape, then the lemma implies that all cross sections that do not include a face of one of the cuboids intersect the same number of cuboids.

We believe that this property is also present in the discrete, 2-increasing setting, but we have not been able to prove this. Specifically, if we say that a 2-increasing, discrete sequence of triples is extremal if it is of length $n^\alpha$ where $\alpha$ is the maximal exponent (ie $F(n)=n^{\alpha+o(1)}$), then we conjecture that the following holds. There exists a function $C:\mathbb{N}^3\mapsto \mathbb{N}$ such that if $T$ is an extremal 2-increasing sequence of triples from $[r]\times[s]\times[t]$ then the number of triples in the plane $(\ast,\ast,z)$ is equal to $C(r,s,t)$, the number of triples in the plane $(\ast,y,\ast)$ is equal to $C(s,t,r)$ and the number of triples in the plane $(x,\ast,\ast)$ is equal to $C(t,r,s)$.


\subsection{Long 2-comparable sequences} \label{long2cs}

We begin with a very short but somewhat abstract argument that there are 2-comparable collections of triples that have length greater than $n^{3/2}$. The argument starts with the following example, given in its grid representation.

\[\young({~3~~4,3~~4~,~~~12,~14~~,1~2~~})\]

This lives in the set $[5]\times[5]\times[4]$ and contains ten triples. Since $10=(5\times 5\times 4)^{1/2}$, this is not yet a suitable example. However, even the tiniest improvement would turn it into an example of what we want, since the number of triples is equal to the bound we are trying to improve. 

This is where looking at the continuous problem helps. We cannot make a ``tiny" improvement to a discrete example, but if we think of this set as a continuous example made out of unit cubes, then Lemma \ref{extremal} implies that we \emph{can} improve it, since the number of cubes in each layer is not constant; labels 1 and 4 appear three times each, while labels 2 and 3 appear only twice each. Then Lemma \ref{equiv} allows us to convert our improved example back into a (much larger) discrete example that exhibits a similar improvement.

Rather than pursuing the above argument in detail, we shall use similar ideas to obtain better bounds and smaller examples. This time our starting point is the following length-five 2-comparable collection of tuples from $[3]^3$:
\[(1,1,1),(1,2,3),(2,3,1),(3,1,2),(3,3,3).\]
In the grid formulation, this is given by
\[\young({~13,3~~,1~2})\]

Interestingly, this example is \emph{not} on the boundary, since $5<3^{3/2}=5.196...$. However, these two numbers are sufficiently close that by optimizing the corresponding continuous example one can still beat the power 3/2, and that gives the best bound we currently know.

Let us therefore convert the example to the continuous variant by viewing it as a union of five $\frac13\times \frac13\times \frac13$ cuboids living inside $[0,1]^3$. We now perform a distortion so that the cuboids have different sizes. Specifically, we shall simultaneously stretch and shrink the cubes by choosing some $x\in(0,1/2)$ and dividing each copy of $[0,1]$ into the three intervals $(0,x), (x, 1-x)$ and $(1-x,x)$. (Symmetry considerations show easily that we are not losing any important flexibility by doing this.) We shall then optimize ${x}$. 

For the resulting collection of cuboids $B$ we have 
\[\|B\|_{1/2}= 2(x^3)^{1/2}+3(x^2(1-2x))^{1/2}\]
which is optimized at $x=(7+\sqrt{5})/22=0.419\dots$ giving
\[\|B\|_{1/2}=\sqrt{\frac{13}{22}+\frac{5 \sqrt5}{22}}=1.048\dots>1.\]
This shows already that $\theta>1/2$, but we can work a little more and obtain a concrete lower bound on $\theta$.

Note first that
\[\|B\|_\alpha = 2x^{3\alpha}+3(x^{2\alpha}(1-2x)^\alpha,\]
so we want to find $\alpha$ as large as possible such that
\[\sup_{x\in(0,1/2)} (2x^{3\alpha}+3x^{2\alpha}(1-2x)^{\alpha})\ge 1.\]
The best we can do here is a numerical calculation, which reveals that the optimal $\alpha$ lies between $0.5154$ and $0.5155$. Therefore $\theta$, the best possible exponent, is greater than $0.5154$.

Applying Lemma~\ref{equiv} we instantly deduce Theorem~\ref{UB}.

It may be of interest to see some small examples of sequences breaking the $n^{3/2}$ bound, since it is not immediately obvious how to extract simple ones from Lemma~\ref{equiv}. We shall now give two, and explain a little how they were constructed.

The process for constructing explicit counterexamples with small $n$ essentially follows the proof of the full upper bound, but we avoid the complexity of Lemma~\ref{equiv} by discretizing the continuous example above in a simple way. We simply subdivide all three dimensions equally and place discrete sequences inside each of the continuous cuboids in the resulting grid. In general this may not work, since we may not be able to fit long sequences inside the cuboids if their shapes are too different. However, a judicious choice of the parameter $x$ in the continuous construction outlined above allows us to keep the cuboid dimensions in a good range.

For example, we can take the above construction but modify it by taking the sub-optimal $x=4/9$. This value of $x$ is chosen because our calculations above showed that we wanted $x>1/3$, and if we take $x$ to be a rational with small denominator we can subdivide coarsely and obtain a discrete sequence that lives inside a small grid. So we subdivide each dimension into 9 sections and scale up by a factor of 9 so that our subdivisions are into unit intervals, and we end up with the following:

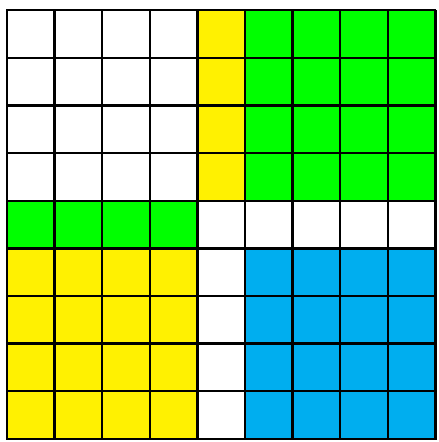
\[
\begin{ytableau}
~ &&&& *(yellow)&*(green)&*(green)&*(green)&*(green) \\
&&&& *(yellow)&*(green)&*(green)&*(green)&*(green) \\
 &&&& *(yellow)&*(green)&*(green)&*(green)&*(green) \\
 &&&& *(yellow)&*(green)&*(green)&*(green)&*(green) \\
 *(green)&*(green)&*(green)&*(green)&&&&& \\
 *(yellow) & *(yellow)& *(yellow)& *(yellow)&&*(cyan)&*(cyan)&*(cyan)&*(cyan) \\
 *(yellow) & *(yellow)& *(yellow)& *(yellow)&&*(cyan)&*(cyan)&*(cyan)&*(cyan) \\
 *(yellow) & *(yellow)& *(yellow)& *(yellow)&&*(cyan)&*(cyan)&*(cyan)&*(cyan) \\
 *(yellow) & *(yellow)& *(yellow)& *(yellow)&&*(cyan)&*(cyan)&*(cyan)&*(cyan) \\
\end{ytableau}
\]
where the yellow blocks correspond to cuboids with third dimension $(0,4)$, the blue block corresponds to a cuboid with third dimension $(4,5)$, and the green blocks correspond to cuboids with third dimension $(5,9)$.

In order to convert this into a discrete sequence, we simply need to fill the cuboids with large 2-comparable collections of tuples. For instance, the bottom left cuboid is $(0,4)\times (0,4)\times (0,4)$, so we want to treat it as a $4\times 4$ grid with labels from the set $[4]$. We can fit $4^{3/2}=8$ labels inside here. Similarly we can fit eight labels in the top right green block, and four labels in each of the other three blocks. This gives us a 2-comparable sequence of triples in $[9]^3$ of length 28, shown in Figure~\ref{Fig2} alongside a smaller example of a 2-comparable sequence of tuples that beats the $(rst)^{1/2}$ bound.

\begin{figure*}[t!]
	\centering
	\begin{subfigure}[b]{0.4\textwidth}
		\centering
		\[
		\young({~~~~4~~89,~~~~3~~67,~~~~289~~,~~~~167~~,6789~~~~~,~~34~~~~5,~~12~~~5~,34~~~~5~~,12~~~5~~~})
		\]
		\caption{A length 28 2-comparable sequence of tuples in $[9]^3$, given in grid formulation. Observe that $28=9^{1.516...}>9^{3/2}$.}
	\end{subfigure}%
	~~~~~~~~~
	\begin{subfigure}[b]{0.4\textwidth}
		\centering
		\[\young({~~134,34~~~,~1~~2,1~~2~})\]
		\caption{Another example of a 2-comparable sequence given in grid formulation. Note that the length of the sequence is 9, while $rst=80<9^2$.}
	\end{subfigure}
	\caption{Some long $[3,2]$ sequences.}
	\label{Fig2}
\end{figure*}

An important observation is that all the sequences discussed in this section are disastrously far from being transitive. For example, in the coloured grid above we see that any choice of three labels from the leftmost green block, the rightmost yellow block and the blue block form an intransitive loop. As a result these constructions pose no problems for Conjecture~\ref{IncUB}.

It is also worth remarking that the best constructions above all began from the same starting point; namely the sequence of length 5 presented at the start of the section. We could hope that there are other short sequences to start from which could yield even better constructions. However, all the sequences that we have tried have yielded significantly worse bounds than the one above. This leads us to think that the optimal $\alpha$ that we approximated earlier has a chance of being the correct exponent for the $[3,2]$ problem.

\section{Related conjectures}\label{waffle}

In this section we shall discuss several conjectures, some of them closely related to well-known questions, that would imply power bounds for the $[3,2]$ or $(3,2)$ problems.

\subsection{Weakening the 2-comparability condition} \label{weaker}

We have not yet been able to improve on Loh's upper bound in the $[3,2]$ case, and a power type improvement here is highly desirable. In this subsection we shall consider how far we can weaken the 2-comparability condition and still have some hope of a non-trivial power-type upper bound, since this may help in the search for a proof.

Recall that the grid representation of a $[3,2]$ subset of $[n]^3$ is a subset $G\subset[n]\times[n]$ with its points given labels from $[n]$ in such a way that the following two conditions are satisfied.

	\textbf{Condition 1.} The labels increase along rows and up columns.
	
	\textbf{Condition 2.} Each label occupies a 2-increasing set of points from the grid.

\noindent Can we weaken these conditions without obviously allowing $G$ to have size $n^{2-o(1)}$?

One weakening that goes too far is simply to omit Condition 2. In this case we can label $(a,b)$ with the label $a+b-n/2$ provided $n/2<a+b\le 3n/2$, which allows us to place about $3n^2/4$ labels. 

If Conditions 1 and 2 hold, and a labelled point $P$ is in the same row as a labelled point $Q$ and the same column as a labelled point $Q'$, then $Q$ and $Q'$ must have different labels. That is because otherwise if $Q$ is to the right of $P$ and $Q'$ is above $P$, then Condition 2 is violated, if $Q$ is to the left of $P$ and $Q'$ is above $P$, then Condition 1 is violated, and the other two cases are similar. Let us give a name to this consequence.

	\textbf{Condition 3.} Given any point $x\in G$, no point in the same row as $x$ (excluding $x$ itself) can share a label with a point in the same column as $x$.

Another weakening we might consider is to replace Condition 1 by Condition 3. However, if we associate matchings with the labels in an obvious way, Condition 3 is saying that these matchings are induced, so we can use the standard Behrend example to show that there are labelled sets of size $n^{2-o(1)}$ that satisfy Conditions 2 and 3. Indeed, take a set $A\subset [n]$ of size $n^{1-o(1)}$ that contains no arithmetic progression of length 3, and label the cells $(x,y)$ on the line $x+y=a\in A$ with the label $z=x-y$ provided that $x-y>0$. In this way we label $n^{2-o(1)}$ cells, and it is easy to check that the labelling satisfies the two conditions.

However, a small strengthening of Condition 3 rules out Behrend-type constructions and leaves the possibility of a power bound wide open. We first give a definition.

\begin{definition}\label{CP}
We denote by $S(c)$ the collection of cells from $G$ with label $c$. We call such sets \emph{label sets}. We also write $P(c)$ for the set of cells in $[n]^2$ that share both a row and a column with a cell from $S(c)$. We call $P(c)$ the \emph{completion} of $S(c)$. 
\end{definition}

The reason for the word ``completion" is that $S(c)$ can be thought of as a matching, and $P(c)$ can be thought of as the smallest complete bipartite graph that contains~it.

Condition 3 is equivalent to the statement that for all labels $c$, the cells in the set $P(c)\setminus S(c)$ are all empty.

For our new variant, we replace Condition 3 with the following stronger condition, which we call Condition 4. It states that Condition 3 holds and additionally that there is no cell $x$ in $G$ such that there exist labels $c$ and $d$ with $c$ appearing to the left of $x$ in the same row and $d$ appearing to the right, and $d$ appearing below $x$ in the same column and $c$ appearing above. 

This condition rules out the following configuration appearing in a subgrid of $G$, where asterisks denote cells which may be labelled or empty:
\[\young({~c*,c~d,*d~})\]
A more appealing way to state the condition is as follows. We first extend Definition \ref{CP}.

\begin{definition}\label{CPext}
	Define the \emph{upper completion} of a label set $S(c)$ to be the set $P_1(c)$ of cells from $G$ that have points labelled $c$ both directly below them and directly to the right, and the \emph{lower completion} $P_2(c)$ to be the set of cells from $G$ that have points labelled $c$ both directly above and directly to the left.
\end{definition} 

We have $P(c)=P_1(c)\cup P_2(c)\cup S(c)$, as illustrated in Figure \ref{Fig3}.
\begin{figure}
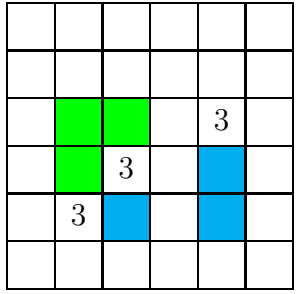

	\[
	\begin{ytableau}
	 ~& &&&& \\
	 &&&&& \\
	& *(green)&*(green)&&$3$& \\
	& *(green)  &$3$&&*(cyan)& \\
	&3  &*(cyan)&&*(cyan)& \\
	&  &&&& ~\\
	\end{ytableau}
	\]
	\caption{The upper and lower completions of a label set as given in Definition \ref{CPext}. $P_1(3)$ is highlighted in green and $P_2(3)$ in blue.}
	\label{Fig3}
\end{figure}

\noindent This gives us the following way of stating the condition.

	\textbf{Condition 4.} Given any two labels $c$ and $d$, the sets $P_1(c)$ and $P_2(d)$ are disjoint.

\begin{problem}\label{weak}
	Let $G\subset[n]\times [n]$ be labelled with points from $[n]$. Suppose that the labelling satisfies Conditions 2 and 4. How many labelled cells can $G$ contain?
\end{problem}

Since Condition 4 is strictly weaker than Condition 1, this is a weakening of the $[3,2]$ problem, so we cannot hope for a power bound as strong as $n^{3/2}$. However, the construction based on Behrend's AP3-free set does not come close to satisfying the two conditions, and a non-trivial power bound seems quite plausible.

\begin{conjecture}
	There exists $\epsilon>0$ such that any labelling satisfying the conditions of Problem \ref{weak} has at most $n^{2-\epsilon}$ labels.
\end{conjecture}

A somewhat different weakening of the $[3,2]$ problem can be obtained from the following observation. Given a subset $G\subset[n]^2$, we can regard it as a bipartite graph with copies of $[n]$ as its vertex sets. If we now assign labels to $G$, we can think of it as a labelled bipartite graph.

\begin{proposition}
If the labelling of $G$ corresponds to a 2-comparable set of triples, then the labelled bipartite graph just described contains no cycle with a sequence of labels that repeats itself twice. That is, there is no cycle of length $2k$ such that as you go along the edges, the sequence of labels is of the form $c_1c_2\dots c_kc_1c_2\dots c_k$.
\end{proposition}

\begin{proof}
Suppose that a repeating cycle of this kind exists. In this bipartite-graphs formulation, Condition 2 says that no two edges with the same label can share a vertex or cross each other (if we imagine that the vertices are arranged in increasing order in two parallel rows). For each edge in the cycle, call it a \emph{left edge} if it occurs to the left of its opposite counterpart (more formally, the vertices connected by the edge $e_i$ are smaller than the vertices connected by the edge $e_{i+k}$, where addition is mod $2k$), and otherwise a \emph{right edge}. 

There must be some $i$ such that $e_i$ is a right edge and $e_{i+1}$ is a left edge. Without loss of generality $e_i$ is the edge $xy$ and let $e_{i+1}$ be the edge $x'y$. Then if $k$ is even the edges $e_{i+k}$ and $e_{i+k+1}$ take the form $zw$ and $z'w$, where $w$ is both smaller than $y$ and greater than $y$, a trivial contradiction. If $k$ is odd, then they take the form $zw$ and $zw'$. This time our assumptions give us that $x>z$, $x'<z$, $y>w$ and $y<w'$. The first two inequalities imply that $c_i>c_{i+1}$, by Condition 1, and the third and fourth imply that $c_i<c_{i+1}$, again giving a contradiction.
\end{proof}

Call a cycle of the kind discussed in the proposition above a \emph{repeating cycle}.

\begin{problem}
Let $G$ be a bipartite graph with two vertex sets of size $n$ and suppose that its edges can be labelled with $n$ labels in such a way that there are no repeating cycles. How many edges can $G$ have? In particular, is there an upper bound of $O(n^{2-\epsilon})$ for some positive $\epsilon$?
\end{problem}

\noindent Note that the problem above does not say anything about orderings on the vertex sets or the set of labels, so it is a weakening to a more ``purely combinatorial" problem. As the proposition shows, a positive answer to the last question would give a non-trivial power bound for the $[3,2]$ problem.

\subsection{Connections to extremal problems for hypergraphs}

In this section we give one last perspective on the $[3,2]$ problem, and give a connection to widely studied  problems about hypergraphs, as well as to a well-known problem of Ruzsa \cite{Ruzsa}.

Let $G$ be a tripartite, 3-uniform, linear hypergraph with vertex sets $X=Y=Z=[n]$. 

\begin{definition}
	We say that $G$ is $(u,v)$-free if there is no collection of $v$ edges spanned by at most $u$ vertices.
\end{definition}

\begin{problem}\label{Turan}
	What is the maximal size of $G$ if it is $(u,v)$-free? In particular, when can we beat the trivial bound of $\Omega(n^2)$?
\end{problem}

This problem (and its generalization to $r$-uniform hypergraphs) has been studied by a large number of people, beginning with Brown, Erd\H os and S\' os~\cite{BES} who proved that $(u,u-2)$-free hypergraphs could contain $\Omega(n^2)$ edges. The well-known ``$(6,3)$ theorem" of Ruzsa and Szemer\'edi~\cite{RS} was the next breakthrough, proving that $(6,3)$-free hypergraphs can contain at most $o(n^2)$ edges, and the Behrend construction~\cite{Behrend} showed that this is almost tight in the sense that there are $(6,3)$-free hypergraphs containing $n^{2-o(1)}$ edges.

The following conjecture of Brown, Erd\H os and S\'os has been open since 1971.

\begin{conjecture}\label{BES}
	If $G$ is $(u,u-3)$-free then it contains $o(n^2)$ edges.
\end{conjecture}

The next result shows how these questions are related to our problem.

\begin{proposition}\label{connect}
	Given a collection of triples $T$, regard it as a tripartite 3-uniform hypergraph $G(T)$ in the obvious way. If $T$ forms a 2-increasing sequence then $G(T)$ is $(9,5)$-free, and if $T$ is a 2-comparable set then $G(T)$ is $(10,6)$-free.
\end{proposition}

\begin{proof}
	Define $F(r,s,t)$ (respectively $G(r,s,t)$) to be the maximum length of a 2-increasing (respectively 2-comparable) sequence of triples in $[r]\times[s]\times[t]$. In order to prove the proposition, we need to show that $F(r,s,t)\le 4$ whenever $r+s+t\le 9$ and $G(r,s,t)\le 5$ whenever $r+s+t\le 10$.
	
	To show that $F(2,3,4)\leq 4$ (and even that $G(2,3,4)\leq 4$), note that if two triples $(a,b,c)$ and $(a',b',c)$ share a third coordinate $c$, then there can be no triples beginning $(a,b')$ or $(a',b)$, which implies that there are at most four triples (since no two triples can share two coordinates). But if all the triples have distinct third coordinates then again there are at most four triples.

	Furthermore, $F(1,4,4)$ is bounded by $1\times 4=4$ trivially, and similarly for $F(1,3,5)$ and $F(2,2,5)$ (and again the same bounds hold for $G$). So to prove the first statement it remains to bound $F(3,3,3)$.
	
	This is a little more difficult. If any coordinate takes the same value three times, then there can be at most three triples, since the other two coordinates must be 11, 22 and 33, which between them rule out all other possibilities for those two coordinates. If in some coordinate at most one value occurs twice, then trivially there are at most four triples. 
	
	So we may assume that in each coordinate two values occur twice. In the grid representation, we are labelling points in $[3]^2$ with labels from $[3]$, and we may assume that two labels appear twice. Since the Cartesian products associated with these two label sets are $2\times 2$ subgrids of the $3\times 3$ grid, they must intersect in a cell. So up to symmetry we have the following configuration:
	
		\[
		\begin{ytableau}
		*(red)&a&b\\
		b&*(yellow)&*(red)\\
		a&*(red)&*(green)\\
		\end{ytableau}
		\]
\noindent where the cells in red cannot be filled as they are part of a Cartesian product associated with a label set. The cell in yellow cannot be filled since $b>a$ and labels must increase up columns and along rows. Finally, we see that the cell in green cannot be filled without violating transitivity.
	
	In order to bound $G(r,s,t)$ by 5 for $r+s+t\le 10$ we need to consider the $(r,s,t)$ combinations $(2,3,5),(2,4,4)$ and $(3,3,4)$, since if $1=r\leq s\leq t$ then the trivial bound of $rs$ suffices. The first case is easy, since if we label all six points of the grid $[2]\times [3]$, then all the labels have to be distinct, which they cannot be if we have only five lables. For the second case, we consider labelling points in $[4]\times[4]$ with labels from $[2]$; if a label is used four times then we cannot fill in any more points, and if each label is used three times then we get two associated $3\times 3$ Cartesian products, which must intersect in a $2\times 2$ subgrid. But any $2\times 2$ subgrid of the Cartesian product associated with one of the label sets actually contains a labelled point, which contradicts Condition 3 of Subsection \ref{weaker}. 
	
	So it remains to check $G(3,3,4)$. Here it is easiest to imagine labelling $[3]\times [3]$ from $[4]$. If at most one label is used twice we are done, and if any label is used three times we are done. So once again we may assume that two labels are used twice, and we once again arrive at the configuration
		\[
		\begin{ytableau}
		*(red)&a&b\\
		b&*(yellow)&*(red)\\
		a&*(red)&\\
		\end{ytableau}
		\]
	
\noindent where the red and yellow cells are unlabelled for the same reason as before. So there can be at most one more label, and $G(3,3,4)=5$ as desired.
\end{proof}

Of course we are interested in a power bound, but because both the $(9,5)$ and $(10,6)$ cases of Problem \ref{Turan} are imposing stronger conditions than those in Conjecture \ref{BES}, it is reasonable to hope that such a bound could hold. Indeed, if Conjecture~\ref{BES} is true then it would seem highly likely that a stronger bound should be possible in the $(u,u-4)$ cases. 

\begin{conjecture} \label{powerbound}
For every $u$ there exists $\epsilon>0$ such that a $(u,u-4)$-free hypergraph with $n$ vertices has at most $O(n^{2-\epsilon})$ edges.
\end{conjecture}

As is also the case with the weakenings discussed at the end of the previous subsection, Proposition \ref{connect} loses some of the strength of the 2-increasing and 2-comparable conditions, so it is quite possible that Conjecture \ref{powerbound} is false, but that a non-trivial power bound still holds for the $[3,2]$ problem.

The strongest known result in the direction of Conjecture \ref{BES} is the following theorem of S\'ark\" ozy and Selkow~\cite{SS}, again stated only in the 3-uniform setting.
\begin{theorem}\label{SS}
	If $G$ is $(v+2+\lfloor \log_2 v\rfloor,v)$-free then $G$ contains $o(n^2)$ edges. In particular, if $G$ is $(8,4)$-free, $(9,5)$-free or $(10,6)$-free then we have an upper bound of $o(n^2)$.
\end{theorem}

This of course directly implies a bound of $o(n^2)$ for both the $(3,2)$ and $[3,2]$ problems, but the proof of Theorem \ref{SS} uses the regularity lemma and consequently does not improve on the $n^2/\exp(\Omega(\log^*n))$ bound from Loh. However, the result above is unlikely to be best possible. Indeed it is easy to prove a bound of $n^{3/2}$ for the $(8,4)$ case -- this follows from Theorem 1.2 of Loh~\cite{inctrips}. From this a power bound for the $(10,5)$ case follows easily, but $(9,5)$ and $(10,6)$ are still out of reach.

An additional reason to try to improve the $(9,5)$ bound to one of power-type is that it would answer the following question of Ruzsa~\cite{Ruzsa}.

\begin{problem}\label{R}
	Let $A\subset [n]$ be a set containing no non-trivial solutions to the equation $2x+2y=z+3w$ (meaning all solutions have $x=y=z=w$). How large can $A$ be as a subset of $n$? In particular, can it have size $n^{1-o(1)}$?
\end{problem}

The equation $2x+2y=z+3w$ is the simplest example of one for which simple Cauchy-Schwarz arguments do not work, but neither does the Behrend construction. So the best known upper bound is obtained by arguments similar to the proofs of Roth's theorem, but the best known lower bounds are of the form $n^\alpha$ with $\alpha<1$. It would be of great interest to know which bounds are nearer to the truth.

\begin{proposition}
	A power bound for the $(9,5)$-free version of Problem \ref{Turan} implies a power bound for Problem \ref{R}. 
\end{proposition}

\begin{proof}
	Suppose that we have a set $A\subset [n]$ with no solutions to the equation $2x+2y=z+3w$. Then define a 3-uniform, linear, tripartite hypergraph $G$ with vertex sets $X=Y=Z=[n]$ and all the faces $(x,y,z)$ with $x-y=z$ and $x+y\in A$. By translating $A$ if necessary (modulo $n$) we can ensure that $G$ has $\mathcal{O}(n|A|)$ faces.
	
	It is easiest to visualise the graph $G$ as a labelled bipartite graph $G'$ on $X\times Y$, where the label on the edge $(x,y)$ is $z$ (for each face $(x,y,z)$ of $G$). Since $G$ is linear, each edge in $G'$ has precisely one label. For a string $S$ of letters, we say that the graph is \emph{$S$-free} if there is no path through $G'$ where the edges have labels following the pattern of $S$. We call such a path an \emph{$S$-path}. For instance, the graph is $aa$-free if no two incident edges have the same label, which follows from the fact $G$ is linear.
	
	Since $A$ is AP3-free, it follows that $G'$ is $aba$-free. We will now show that $G'$ is $abcab$-free, and we shall further show that any configuration of 5 faces supported on 9 vertices gives rise to either an $aa$-path, and $aba$-path or an $abcab$-path. 
	
	First we show that $G'$ is $abcab$-free. An $abcab$-path has 5 edges $(f_1,f_2,f_3,f_4,f_5)$. Each $f_i$ is an edge $(x_i,y_i)$ with $x_i+y_i=a_i\in A$. Without loss of generality we have $x_1=x_2$, $y_2=y_3$, $x_3=x_4$ and $y_4=y_5$. The constraints on the labelling then translate into arithmetical constraints on the $a_i$ and we find that
	\[a_4=a_1+2(a_2-a_3)\]
	and
	\[a_5=a_2+2(a_1-a_3).\]
	But then 
	\[2(a_4+a_3)=2a_1+4a_2-2a_3=a_5+3a_2\ ,\]
	which cannot happen for $a_i\in A$.
	
	We now claim that any configuration of five faces of $G$ supported on nine vertices gives rise to either an $aa$-path, an $aba$-path or an $abcab$-path in $G'$. This is little more than a brute-force check -- we need to confirm that a bipartite graph $H$ on $U\times V$ with five edges, each labelled from $W$, has one of the required paths if $|U|+|V|+|W|\le 9.$ 
	
	It obviously suffices to check the case $|U|+|V|+|W|= 9.$ If $|W|=1$ then $H$ must be a matching to avoid $aa$-paths, and so if $H$ has five edges then $|U|+|V|\ge 10$. If $|W|\ge 5$ then $|U|+|V|\le 4$ and so $H$ cannot have more than four edges. So we may assume $2\le |W|\le 4$.
	
	If $|W|=2$, then to avoid $aba$-paths and $aa$-paths $H$ must be a union of components of size at most 2. The only non-trivial case (without loss of generality) is $|U|=3$ and $|V|=4$, and if $H$ has five edges we find that two vertices of $U$ must have degree 2 with disjoint neighbourhoods in $V$, meaning that the final fifth edge cannot exist without connecting the components.
	
	If $|W|=4$ we must have $|U|=2$ and $|V|=3$ without loss of generality. All but one edge is present. Therefore there is a vertex from $U$, say $u_1$, with all three vertices from $V$ as neighbours. The other vertex from $U$, say $u_2$ has only two neighbours from $V$, say $v_1$ and $v_2$. The edges from $v_1$ and $v_2$ to $u_2$ must both be labelled with the label that $u_1$ does not see in order to avoid $aa$-paths through $u_1$. But this creates an $aa$-path through $u_2$.
	
	The remaining cases are $|U|=|V|=|W|=3$ and $|U|=2,|V|=4, |W|=3$. The latter case is dealt with by observing that there must be a vertex $u_1$ that sees all three labels, and the neighbourhood of the other vertex $u_2\in U$ intersects the neighbourhood of $u_1$ in at least one vertex, say $v$. But since $u_2$ sees two labels we get an $aba$-path with middle edge $u_1v$.
	
	Now let us consider the case $|U|=|V|=|W|=3$. If any vertex has degree 3 we will get an $aa$-path or an $aba$-path, so we may suppose that $U=\{u_1,u_2,u_3\}$ and that $u_1$ and $u_2$ have degree 2 and $u_3$ has degree 1. Clearly $u_1$ and $u_2$ must have a common neighbour $v_1$. Without loss of generality let $u_1v_2$ and $u_2v_3$ be edges of $H$. There is only one labelling of these four edges that avoids $aa$-paths and $aba$-paths, so we may assume without loss of generality that $u_1v_2$ and $u_2v_3$ have label $w_1$, $u_1,v_1$ has label $w_2$ and $u_2v_1$ has label $w_3$. Now there is one edge from $u_3$. It cannot go to $v_1$ without creating a vertex of degree 3, so without loss of generality we have the edge $u_3v_2$. This must be labelled $w_3$ to prevent $aa$-paths and $aba$-paths, but this leaves us with an $abcab$-path.
	
	Putting this together, since $G'$ is $aa$, $aba$ and $abcab$-free we find that we have been able to use $A$ to build a linear tripartite hypergraph $G$ which is $(9,5)$-free and has size $\mathcal{O}(n|A|)$. The result follows.
\end{proof}

The above proposition can be regarded either as a strong motivation for trying to find a power-type improvement to the trivial bound for the $(9,5)$ problem or as a lower bound on the difficulty of doing so.
\section{The Generalized $[r,s]$ Problem}

In this section we shall discuss the generalization to $s$-increasing or $s$-comparable sets of $r$-tuples, which we mentioned in the introduction. 

In the introduction we gave a construction of an $[r,s]$-sequence of size $n^{r/s}$ and commented that by an easy probabilistic argument it is not generally sharp. Here is that argument in more detail.

\begin{lemma}\label{FMM}
	Let $n$ be fixed and let the ratio $\beta=s/r$ be fixed with $\beta<(1-1/n)/2$. Then $F_{r,s}(n)$ (and hence $G_{r,s}(n)$ also) grows exponentially with $r$.
\end{lemma}

\begin{proof}
	We use a simple first-moment argument, with modification. Let us choose a sequence of size $S$ of $r$-tuples by selecting the digits of each tuple from $[n]$ uniformly and independently at random. 
	
	Then if we take two distinct tuples $x_i$ and $x_j$ (with $i<j$) from this collection, the number of coordinates in which the first is larger than the second is binomially distributed as the sum of $r$ independent Bin($(1-1/n)/2$) distributions. The probability that $x_i$ is not $s$-less than $x_j$ is at most the probability this binomial distribution takes a value less than $s=\beta r$. This event is a binomial tail probability and consequently is exponentially small in $r$. 
	
	Therefore by taking $S$ to be exponentially large, and by removing any $x_i$ that forms part of a pair that do not have the appropriate $s$-increasing relation, we find an exponentially large $(r,s)$-sequence.
\end{proof}

As discussed in the introduction, the trivial upper bound for the size of an $s$-comparable set of $r$-tuples from $[n]$ is $n^{r-s+1}$, and there is a natural construction of size $n^{r/s}$. In the earlier sections of the paper we concentrated on the problem of fixing $r$ and $s$ (as 3 and 2 respectively) and aiming to improve one or other of these bounds. This problem seems to be difficult, even for larger fixed values of $r$ and $s$ where the bounds can be very far apart indeed. We shall therefore concentrate on the regime where $n$ is fixed and $r$ and $s$ vary but have a fixed ratio. We shall discuss the comparable version rather than the increasing version, but only since we have not found an interesting difference between the two problems in this regime.

Specifically, we will study the following problem.

\begin{problem}\label{genP}
	Let $n$ be a fixed positive integer, and $0<\beta<1$ a fixed  real number. Let $H_{n,\beta}(r)$ be the maximal size of an $s$-comparable collection of $r$-tuples, where $s=\beta r$. For fixed $n$ and $\beta$, how does $H_{n,\beta}(r)$ grow with $r$?
\end{problem}

Lemma~\ref{FMM} tells us that if $\beta<(1-1/n)/2$, then $H_{n,\beta}(r)$ grows exponentially with $r$. We complement this lemma with the following result.

\begin{proposition}\label{DRC}
	If $\beta>(1-1/n)/2$, then $H_{n,\beta}(r)$ is bounded.
\end{proposition}

Before proving Proposition \ref{DRC}, we note a parallel with a problem concerning real vectors.

\begin{problem}\label{vecs}
	Let $\eta\in[-1,1]$ be a fixed real. Then how large may a collection $V=\{v_i\}$ of $d$-dimensional real unit vectors be if $V$ has the property that for all $i\neq j$ we have $\langle v_i, v_j\rangle\leq\eta$? 
\end{problem}

\noindent This problem is well understood~\cite{Bollobook}. In particular, it is known that when $\eta$ is positive then the maximum size of $V$ grows exponentially in $d$, while when $\eta$ is negative the maximal size of $V$ is bounded independently of $d$ (by about $-\eta^{-1}$). When $\eta=0$ then $|V|\le 2d$, given by choosing an orthonormal basis $\{e_i\}$ and taking $V=\{e_i\}\cup\{-e_i\}$.

The parallels with Problem \ref{genP} are quite clear. In both problems we have a collection of objects constrained by some condition of pairs from the collection, and we have a parameter with a threshold value on one side of which the size of the collection may be exponentially large and on the other side of which the size of the collection is bounded. Moreover, the threshold is the expected value of the parameter when the two objects are chosen at random. It is tempting to conclude that Problem~\ref{genP} can be tackled by cleverly identifying tuples of integers with vectors in a way that translates Problem~\ref{genP} into Problem~\ref{vecs}, but the authors were unable to find such a transformation. 

Furthermore, there are some reasons to think that a transformation of this kind does not exist. In the unit-vectors problem, if we want to deduce that the size of $V$ is bounded when $\eta$ is negative, it is enough to assume not that \emph{every} inner product is at most $\eta$, but merely that the \emph{average} inner product is at most $\eta$. However, a similar weakening of the hypotheses for our problem in case (iii) is no longer sufficient for boundedness. Take, for example, the case $n=3$, and for an arbitrary $m$ take a collection of $3m$ $r$-tuples, where $m$ of them are equal to $(1,1,\dots,1)$, $m$ of them are equal to $(2,2,\dots,2)$ and $m$ of them are equal to $(3,3,\dots,3)$. Then if you choose two triples randomly from this collection, the average number of places where they differ is $2r/3$, which is significantly more than $(1-1/n)r/2$. It is easy to modify this example, if one wishes to, to make all the $r$-tuples distinct with a large value of $m$ at only a small cost to the average. 

It therefore appears that we are forced to use a more complicated argument in the proof of Proposition~\ref{DRC}. We shall apply a dependent random selection argument to pass from a collection of $r$-tuples to a large subcollection that resembles one whose members have had their coordinates selected independently at random from some distribution that depends on the coordinate. In an example such as the above, this dependent random selection would tend to pick out a subset that consisted of multiple copies of the same sequence, which would then lead to a contradiction. In the general case, the calculation is more delicate but we obtain a similar contradiction if the number of $r$-tuples we start with is large enough. 

We will begin by quoting three results from a preprint of the first author~\cite{Gowers}.

The first encompasses the dependent random choice aspect of the argument:

\begin{theorem}\label{G1}
	Let $G$ be a bipartite graph with vertex sets $X$ and $Y$ of sizes $m$ and $n$ respectively and let $\delta, \eta, \epsilon$ and $\gamma$ be positive constants less than 1. Let $\delta_1(x,x')$ be the density of the shared neighbourhood of $x$ and $x'$. Suppose that there are at least $\epsilon m^2$ pairs $(x,x')\in X^2$ such that $\delta_1(x,x')\ge \delta(1+\eta)^{1/2}$. Then there is a constant $\alpha\ge \delta$ and a subset $B\subset X$ of density at least $(\epsilon\gamma)^{8\eta^{-2}\log(\delta^{-1})^2}$ with the following two properties.
	\begin{enumerate}
		\item $\delta_1(x,x')\ge \alpha$ for all but at most $\gamma|B|^2$ pairs $(x,x')\in B^2$.
		\item $\delta_1(x,x')\le \alpha (1+\eta)$ for all but at most $\epsilon|B|^2$ pairs $(x,x')\in B^2$.
	\end{enumerate}	
\end{theorem}

The second is a straightforward lemma that translates between different formulations of quasirandomness. We import also the definition of the \emph{box norm}, which is as follows:
$$\|f\|_{\square}^4=\mathbb{E}_{x,x',y,y'}f(x,y)f(x,y')f(x',y)f(x',y').$$
Note that in the preprint~\cite{Gowers} the box norm is referred to as the $U_2$ norm, written $\|.\|_{U_2}$. Recall that in a dense graph $f$ we have $\|f\|_{\square}=\mathcal{O}(1)$.

\begin{lemma}\label{G2}
	 Let $X$ and $Y$ be finite sets and let $f:X\times Y\to \{0,1\}$ be a bipartite graph. Let $\delta_1(x,x')$ be the density of the shared neighbourhood of $x,x'\in X$ and $\delta_2(y,y')$ be the density of the shared neighbourhood of $y,y'\in Y$. Let $\delta_2(y)$ be the density of the neighbourhood of $y\in Y$. Then TFAE.
	 \begin{enumerate}[(i)]
		\item $\mathbb{E}_{x,x'}\big|\delta_1(x,x')-\|\delta_2\|_2^2\big|^2\le c_1\|f\|_{\square}^4.$
	 	\item $\|f-1\otimes\delta_2\|_{\square}\le c_2\|f\|_{\square}^4.$
	 \end{enumerate}
\end{lemma}

The third is obtained from Lemma~\ref{G2}, and will be used to translate the quasirandomness into an applicable condition. 

\begin{lemma}\label{G3}
	Let $X$, $X'$ and $Y$ be finite sets and let $f:X\times Y\to \{0,1\}$ and $f':X'\times Y\to \{0,1\}$ be bipartite graphs. Let $\delta_1$ be the density of the shared neighbourhood of its argument(s) as a subset of $Y$, let $\delta_2(y)$ be the density of the neighbourhood of $y\in Y$ as a subset of $X$ and let $\delta_2'$ be the density of the neighbourhood of $y\in Y$ as a subset of $X'$. Let $0<c\le 2^{-24}$, and suppose that $$\mathbb{E}_{x_1,x_2}\big|\delta_1(x_1,x_2)-\|\delta_2\|_2^2\big|^2\le c\|f\|_{\square}^4$$
	and that $$\mathbb{E}_{x_1',x_2'}\big|\delta_1(x_1',x_2')-\|\delta_2'\|_2^2\big|^2\le c\|f'\|_{\square}^4.$$
	Then 
	$$\mathbb{E}_{x,x'}\big|\delta_1(x,x')-\langle \delta_1,\delta_2\rangle\big|^2\le 16c^{1/16}\|f\|_{\square}^2\|f'\|_{\square}^2.$$
\end{lemma}

\begin{proof}[Proof of Proposition~\ref{DRC}]
	We prove the result by induction on $n$, with $n=1$ being trivial.
	
	Suppose we have a $\beta r$-comparable subset $S\subset [n]^r$. We will first pass to a subset $T$ of $S$ such that each tuple in $T$ has entry $i$ in at least $r/2n^2$ different positions. Indeed, suppose that there exists a subcollection $S'\subset S$ and an entry $i$ such that every tuple from $S'$ has entry $i$ in fewer than $r/4n^2$ positions. Then, by replacing the entries $i$ with $i+1$ (or $i-1$ if $i=n$) and relabelling so that the entries come from $[n-1]$, we get a collection of $r$-tuples with entries from $n-1$ and the $r$-tuples are pairwise at least 
	$$\frac{1-1/n}{2}r-2r/4n^2>\frac{1-1/n}{2}$$
	comparable. By our induction hypothesis the size of $S'$ is therefore bounded independently of $r$. Therefore, for $|S|$ sufficiently large (independent of $r$) we can find the required subcollection $T$ of size proportional to $|S|$ where the constant of proportionality is dependent on $n$ but not on $r$.
	
	We now consider the following set-up. We form a bipartite graphs $G_1=X\times Y$ where $X$ has one vertex for each tuple in $T$ and $Y=[r]$, and the edge $(x,k)$ is present in $G_1$ if the tuple corresponding to $x$ has the entry $1$ in the $k$th position. Write $\delta_1(x,x')$ for the density of the shared neighbourhood of $x$ and $x'$, and $\delta_2(k)$ for the density of the neighbourhood of $k\in Y$.
	
	We will first apply Theorem \ref{G1} to get a certain quasirandomness property for the graph $G_1$.
	
	Since the tuples in $T$ have every possible entry at least $r/4n^2$ times, we see that the degree of each vertex $x\in X$ is at least $r/4n^2$. Consequently the average degree in $Y$ is at least $|S|/4n^2$ and so the number of pairs $x,x'$ from $X$ that share a common neighbour is at least $|S|^2r/16n^4$. We thus find that at least $|S|^2/16n^4$ pairs $x,x'\in X$ have shared neighbourhood of size at least $r/16n^4$.
	
	This allows us to apply Theorem \ref{G1} with $\epsilon\le 1/16n^4$, $\delta(1+\eta)^{1/2}\le 1/16n^4$ and $\gamma=\epsilon$ small. We thus find a constant $\alpha_1\ge \delta$ and a subset $B$ of $X$ of size proportional to $|X|$ (ie $|B_1|=\lambda(\epsilon,\gamma,\eta)|X|$ and $\lambda$ is independent of $r$) such that
	$$\alpha_1\le \delta_1(x,x')\le (1+\eta)\alpha_1$$
	for all but at most $2\epsilon|B_1|^2$ pairs $(x,x')\in B_1^2$.
	
	Now we can define $G_2$ to be the bipartite graph $G_2=X\times Y$ where $X$ has one vertex for each tuple in $B$ and $Y=[r]$, and the edge $(x,k)$ is present in $G_2$ if the tuple corresponding to $x$ has the entry $2$ in the $k$th position. We can repeat the above argument to find a proportionally sized $B_2\subset B_1$ such that 
	$$\alpha_2\le \delta_1(x,x')\le (1+\eta)\alpha_2$$
	for all but at most $2\epsilon'|B_1|^2$ pairs $(x,x')\in B_1^2$. By taking $\epsilon$ much smaller than $\epsilon'$ we can ensure that in $G_1$ we also have
	$$\alpha_1\le \delta_1(x,x')\le (1+\eta)\alpha_1$$
	for all but at most $2\epsilon'|B_2|^2$ pairs $(x,x')\in B_2^2$.
	
	Continuing this for all of the $n$ graphs $G_i$ defined to continue the obvious pattern of $G_1$ and $G_2$ above, we can eventually find (for any $\eta,\mu>0$) a subset $B$ of the set of tuples of size $\lambda(\mu,\eta)|S|$ and constants $\alpha_i>0$ such that in the graph $G_i$
	$$\alpha_i\le \delta_1(x,x')\le (1+\eta)\alpha_i$$
	for all but at most $\mu|B|^2$ pairs $(x,x')\in B^2$.
	
	This tells us that simultaneously all the graphs $H_i$, defined by taking the induced subgraph of $G_i$ on $B\times [r]$, are in a certain unbalanced sense quasirandom. More precisely, since our graphs have the property that almost all pairs of vertices $x,x'\in B$ have shared neighbourhoods of approximately the same size, the LHS of condition (i) from Lemma \ref{G2} is small. Therefore Lemma \ref{G2} tells us that the $H_i$ are quasirandom permutations of the rank 1 matrix $1\otimes \delta_2$. These can be thought of as behaving like random bipartite graphs with a given degree sequence.
	
	We will now apply Lemma~\ref{G3}, which translates the quasirandomness into an applicable condition. For distinct $i$ and $j$ we view the graphs $H_i$ and $H_j$ taking $H_i=X\times Y$ with vertices in $X$ corresponding to the tuples in $B$ and $Y=[r]$ and $H_j=X'\times Y$ with vertices in $X'$ also corresponding to the tuples in $B$. Let $\delta_{i,j}(x,x')$ be the density of the shared neighbourhood of $x\in X$ and $x'\in X'$. By applying Lemma~\ref{G3} we find that for any distinct $i,j$ we have that for almost all pairs $(x,x')\in X\times X'$ the shared neighbourhood $\delta_{i,j}(x,x')$ has density approximately $\mathbb{E}_{x,x'} \delta_{x,i}\delta_{x',j}$, where $\delta_i(x)$ is defined to be the density of the neighbourhood of vertex $x$ in $H_i$ (which is also the density of the number of positions in which the tuple corresponding to $x$ has entry $i$). Specifically, we have 
	$$\mathbb{E}_{x,x'}\bigg|\delta_1(x,x')-\mathbb{E}_{x''}\delta_i(x'')\delta_j(x'')\bigg|<\theta$$
	where $\theta$ can be made arbitrarily small provided $B$ is sufficiently large. 
	
	Observe that for any pair $(x,x')$ of tuples from $B$ we must either have $x<_{\beta r} x'$ or $x'<_{\beta r} x$ since $B$ is $\beta r$-comparable. Therefore we must have
	$$\mathbb{E}_{x\neq x'}\left(\sum_{i<j}\delta_{i,j}(x,x')\right)\ge \beta$$
	and therefore
	$$\sum_{i<j}\mathbb{E}_{x}\delta_i(x)\delta_j(x)\ge \beta -\theta n^2$$
	which gives
	$$\mathbb{E}_{x}\left(\sum_{i<j}\delta_i(x)\delta_j(x) \right)\ge \beta -\theta n^2.$$
	Now we note that we also must have
	$$\sum_i \delta_i(x)=1$$ 
	for all tuples $x\in B$. It is an easy exercise to show that 
	$$\sum_{i<j}\delta_i(x)\delta_j(x)$$
	is maximized subject to the constraint 
	$$\sum_i \delta_i(x)=1$$
	when $\delta_i(x)=1/n$ for all $i$. Therefore
	$$\mathbb{E}_{x}\left(\sum_{i<j}\delta_i(x)\delta_j(x) \right)\le\frac{n(n-1)}{2n^2}$$
	which implies that
	$$\beta -\theta n^2\le\frac{n(n-1)}{2n^2}=\frac{1-\frac1n}{2}.$$
	But if $\beta>(1-1/n)/2$ then by making $\theta$ sufficiently small we have a contradiction. So it must be that we cannot make $\theta$ arbitrarily small, and so $|S|$ is bounded independently of $r$.
\end{proof}

We have not attempted to obtain an explicit bound on the dependence of $H_{n,\beta}(r)$ on $(1-1/n)/2-\beta$. If we were concerned to find as good a bound as possible, then instead of using Theorem \ref{G1} iteratively it would be more efficient to prove directly a version of the theorem that works for $n$ bipartite graphs simultaneously.

The final case left to consider is the threshold $\beta=(1-1/n)/2$. Given the parallel with the vector problem, we expect the size of the collection in this threshold case to be unbounded but sub-exponential -- perhaps even only linear in $s$.

We have not managed to prove this, but in the other direction it is not too difficult to find a linear-sized construction, at least when $n$ is a prime power.

\begin{proposition}
	Let $q$ be a fixed prime power and let $\beta=(1-1/q)/2$. Let $r=q^k$ and $s=\beta r$. Then there exists an $s$-increasing collection of $r$-tuples from $[q]$ of size $q^{k+1}=qr$.
\end{proposition}

\begin{proof}
	Let $\mathbb F_q$ be the field with $q$ elements and consider the set of all affine functions from $\mathbb{F}_q^k$ to $\mathbb{F}_q$, that is to say functions $f$ of the form $f: x\mapsto ax+b$ for $a\in \mathbb{F}_q$ and $b\in \mathbb{F}_q^k$, viewed in the obvious way as a collection of sequences of length $q^k$ with entries from $[q]$.
	
	Note that two distinct affine functions agree on a subspace of codimension 1 or disagree everywhere, so any distinct pair of sequences from this collection agree in at most $q^{k-1}$ places. Additionally, we see that if two sequences of length $q^k$ from $[q]$ agree in at most $q^{k-1}$ places then they are certainly $(q^k-q^{k-1})/2$-comparable, and the proposition follows by taking our collection to be the set of affine functions considered above.
\end{proof}

The remaining goal, therefore, is to establish a sub-exponential bound when $\beta=(1-1/n)/2$. This would establish the desired trichotomy and mirror the behaviour of Problem \ref{vecs}. This appears to be more difficult than the corresponding vector question, and is another possible direction for future work.

\section{Conclusions and future directions}\label{conc}

This paper has raised more questions than it has answered, but we hope that we have given convincing motivation for a rich collection of related and surprisingly challenging problems.

The development that we would most like to see is an improvement to the power bound that we achieve in Theorem~\ref{PB}, ideally to a bound of $n^{3/2}$. If this lower bound is indeed sharp, one would expect it to be provable by a clean inductive argument, but we have had trouble making this work. 

To give an idea of the difficulty, we will describe one possible approach along these lines. Say that a 2-increasing collection of triples $T$ in $[r]\times [s]\times [t]$ has a \emph{decomposition} if we can pick a coordinate, say $t$ without loss of generality, such that the following holds. We can find a partition of $[r]\times[s]$ into sets $R_i\times S_i$ where $R_i\subset [r]$ and $S_i \subset [s]$ for all $i$ such that there exists a partition $[t]$ into sets $T_i$ in such a way that all the triples of $T$ lie in the sets $R_i\times S_i\times T_i$. 

Suppose that all 2-increasing collections of triples have a decomposition. Then we could use induction to bound the total number of triples by $\sum_i(|R_i||S_i||T_i|)^{1/2}$, which by Cauchy-Schwarz is at most $(rst)^{1/2}$, and the problem would be solved. 

It is very tempting to conjecture that a decomposition always exists, since no counterexample is easily found by hand and it would also provide a natural proof of the conjectured power bound for 2-increasing sequences. However, we eventually came across a counterexample, which is given in its three grid representations in Figure~\ref{lastFig}.

\begin{figure*}[t!]
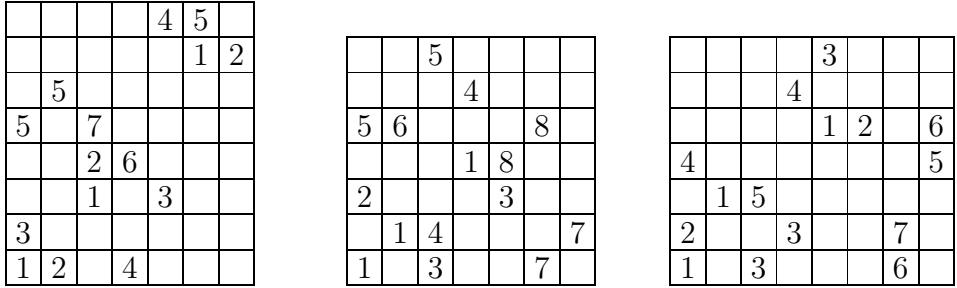

	\centering
	\begin{subfigure}[b]{0.3\textwidth}
		\centering
		\[
		\young({~~~~45~,~~~~~12,~5~~~~~,5~7~~~~,~~26~~~,~~1~3~~,3~~~~~~,12~4~~~})
		\]
	\end{subfigure}%
	\begin{subfigure}[b]{0.3\textwidth}
		\centering
		\[\young({~~5~~~~,~~~4~~~,56~~~8~,~~~18~~,2~~~3~~,~14~~~7,1~3~~7~})\]
	\end{subfigure}%
	\begin{subfigure}[b]{0.3\textwidth}
		\centering
		\[\young({~~~~3~~~,~~~4~~~~,~~~~12~6,4~~~~~~5,~15~~~~~,2~~3~~7~,1~3~~~6~})\]
	\end{subfigure}
	\caption{A counterexample to the decomposition conjecture. It is a 2-increasing sequence of 15 triples from $[6]\times[7]\times[8]$. We give it in grid representation for each of the three possible choices of label coordinate.}
	\label{lastFig}
\end{figure*}

It seems to be hard to find such counterexamples. The example above was found with the help of a computer search. Very briefly, the algorithm we used works as follows. It builds up a 2-increasing sequence triple by triple, and at each step it randomly chooses a minimal triple (in the usual partial order on $[n]^3$) that is 2-greater than all the triples chosen so far, halting when it runs out of possibilities. Then it checks for decomposability.

The check can be done in polynomial time quite easily. Given a grid
representation, we can decompose it in the desired way if we can find a non-trivial partition of the grid into Cartesian products such that no label appears in more than one of the cells of the partition. If two labels occur in some row and also in some column, then the Cartesian product that contains one is forced to contain the other. So the algorithm replaces these two labels by a single label and iterates. If it ends up with just one label, that proves that the example is not decomposable using the label coordinate (and the converse holds too). It then carries out the test for each choice of label coordinate.

Our experiments with this program seem to indicate that indecomposable examples are quite rare, but this may simply be because we have not yet found the right model for random 2-increasing sequences. We have experimented with adding conditions such as choosing at each stage a minimal triple that satisfies an
additional condition. When the program chooses a purely random minimal triple, we stumbled on an example with $n=20$. (More precisely, we stumbled on an example that was almost indecomposable, and could be made indecomposable by
removing a few triples.) This happened only once, and seems to have been quite lucky. Adding the additional condition that the sum of squares of the coordinates is minimized led to the example above -- in this case we set $n=8$ and removed one triple from the randomly generated sequence. 

For large $n$, these random models seem to create 2-increasing sequences of size about $2n$, apart from one model that looks as though it is growing at a rate more like $n^{4/3}$. That model is to take a minimal triple at each stage but to maximize its smallest coordinate (and to make the choice randomly in the case of ties).

For all the examples we know of 2-increasing sequences that attain the bound $(rst)^{1/2}$ it \emph{is} possible to find a decomposition of the kind that could be used for an inductive proof. That leads to the following more precise conjecture.

\begin{conjecture}\label{precise}
Let $T$ be a 2-increasing sequence in $[r]\times[s]\times[t]$. Then $|T|\leq(rst)^{1/2}$, and equality holds only if there is a decomposition of the kind discussed above for some choice of label coordinate.
\end{conjecture}

The non-decomposable example given above not extremal, since $15$ is quite a bit smaller than $(6\times 7\times 8)^{1/2}$, so it does not disprove this conjecture. However, it shows that in order to prove the existence of a decomposition, it is necessary to use the extremality somehow, and it is not obvious how to do that. (This is why we felt that Lemma \ref{extremal} could turn out to be important.) 

It is possible to go one step further than Conjecture~\ref{precise} in the hope of classifying all 2-increasing sequences of length $(rst)^{1/2}$. For this purpose we tentatively formulate the following conjecture, which has survived some small-scale searches for counterexamples.

\begin{conjecture}\label{class}
	Let $T$ be a 2-increasing sequence in $[r]\times[s]\times[t]$. Then $|T|\leq(rst)^{1/2}$, and equality holds if and only if it can be built up as follows:
	\begin{enumerate}
		\item Choose a coordinate, say the third without loss of generality, and partition $[r]\times[s]$ into sets $R_i\times S_i$ where all of the $R_i$ and $S_i$ are intervals. Using the obvious ordering on disjoint intervals (and calling intervals incomparable if they intersect), ensure that the rectangles $R_i\times S_i$ are ordered in a 1-increasing fashion.
		\item Partition $[t]$ into an increasing set of disjoint intervals $T_i$ such that $|R_i||S_i|$ is proportional to $T_i$ (so that we get equality when we apply Cauchy-Schwarz).
		\item Put extremal examples into the sets $R_i\times S_i\times T_i$.
		\item If it is possible to permute two rows, columns or labels while preserving the 2-increasing property (with a different order) then feel free to do so.
	\end{enumerate}
\end{conjecture}

Note that the fourth operation above is necessary for the conjecture to be true. An example that shows why is the sequence given by grid representation
\[\young(~~2,~2~,~~1,2~~,~1~,1~~)\]
which will not decompose using just the first three operations.

It would also be extremely interesting to obtain a non-trivial power-type upper bound for the 2-comparable problem, especially now we know that $n^{3/2}$ is not the right lower bound. Another reason for being interested in this problem is that, as we have shown, it is closely related to some central and quite long-standing problems in extremal hypergraph theory and additive combinatorics.

Finally, there are many interesting generalizations of Loh's original problem, from the minimalist variants and extremal hypergraph problems described in Section \ref{waffle} to the generalized $[r,s]$ problems studied in the previous section. Many of the resulting questions are not yet answered, and some of them look as though they may be approachable. Perhaps the most annoying question to which we do not know the answer is the following.

\begin{question}
Is there a single pair of integers $1\leq s<r$ for which the trivial upper bound of $n^{r-s+1}$ for the size of the largest $s$-comparable subset of $[n]^r$ can be improved by a non-trivial power of $n$?
\end{question}

\end{document}